\documentclass{amsart}
\usepackage{xypic}
\usepackage{epsfig}
\usepackage{tikz}
\usepackage{graphicx}
\usepackage{amsthm, amsmath}
\usepackage{caption}
\usepackage{subcaption}

\newtheorem{theorem}{Theorem}
\newtheorem*{thmx}{Theorem}
\newtheorem{proposition}[theorem]{Proposition}
\newtheorem{corollary}[theorem]{Corollary}
\newtheorem{lemma}{Lemma}
\theoremstyle{definition}
\newtheorem{definition}{Definition}
\newtheorem{example}{Example}
\theoremstyle{remark}
\newtheorem{remark}{Remark}

\numberwithin{equation}{section}

\providecommand{\G}{\Gamma}
\providecommand{\g}{\gamma}

\providecommand{\B}{\mathcal{B}}
\providecommand{\A}{\mathcal{A}}
\providecommand{\bs}{\backslash}

\providecommand{\BR}{\mathrm{Br}}

\begin{document}

\title{Subword complexity and Sturmian colorings of regular trees}

\author{Dong Han Kim, Seonhee Lim}

\thanks{The first author is supported by KRF 2012R1A1A2004473. The second author is supported by KRF 2012-000-8829, KRF 2012-000-2388 andTJ Park Science Fellowship.}

\date{}

\begin{abstract}
In this article, we study subword complexity of colorings of regular trees. We characterize colorings of bounded subword complexity and study Sturmian colorings, which are colorings of minimal unbounded subword complexity.

We classify Sturmian colorings using their type sets. We show that any Sturmian coloring is a lifting of a coloring on a quotient graph of the tree which is a geodesic or a ray with loops possibly attached, thus a lifting of an ``infinte word". We further give a complete characterization of the quotient graph for eventually periodic ones.
\end{abstract}

\maketitle

\section{Introduction}

Let $T$ be a regular tree, i.e. a tree whose degree of vertex is constant, $VT$ be its vertex set and $G=Aut(T)$ be the group of all automorphisms of $T$. 
Let $\A$ be a countable set which will be called the alphabet.

By a coloring of the tree $T$, we mean a vertex coloring, i.e. any map $\phi : VT \to \A$. 
In this article, we define an invariant of a coloring $\phi$ called subword complexity.

One of our motivations for studying subword complexity is to give an invariant of an automorphism of a tree, relative to a discrete subgroup of $G$. 
For example, let $\Gamma$ be the group generated by $k$ generators $a_i, i=1, \cdots, k$ with relations $a_i^2=1$, and $T$ be its Cayley graph. 
To any element  $g$ of $G$ is associated a coloring $\phi_g$ (see Example~\ref{rmk1}).
The coloring $\phi_g$ is periodic if and only if $g$ is an element of the commensurator of $\Gamma$ \cite{LMZ}. 
Commensurators play an important role in studying discrete subgroups of Lie groups and automorphism groups of trees (\cite{M}, \cite{LMZ}, \cite{S}, \cite{R}, \cite{ALN}).
See Section~\ref{subsec:2.1} for details.

With this motivation in mind, we consider unrooted trees, rather than rooted binary trees which were studied in computer science (\cite{BBCF1}, \cite{BBCF2}, \cite{GG}). Theory of subword complexity and Sturmian colorings developed in this article is quite different from that for rooted binary trees \cite{BBCF1}. It seems that there is no direct relation between them.

Unrooted regular trees and their automorphism groups are important objects in geometric group theory in many aspects, as they are Cayley graphs of finitely generated groups and also 1-dimensional buildings, which is a non-archimedean analogue of rank-1 Riemannian symmetric spaces.

For an infinite sequence $u$, the subword complexity $p_u(n)$ is defined as the number of different subwords of length $n$ in $u$.
Hedlund and Morse showed that $p_u(n)$ is bounded if and only if $u$ is eventually periodic \cite{HM}. 
A sequence $u$ is called Sturmian if $p_u(n) = n+1$. (See for example \cite{L} and \cite{Fo} for details on Sturmian words.)

We define subword complexity $b_\phi(n)$ of a coloring $\phi$ as the number of non-equivalent colored $n$-balls in the tree colored by $\phi$. We show that $\phi$  is periodic if and only if its subword complexity $b_\phi(n)$ is bounded. We study Sturmian colorings using the type sets of  vertices. 

The main result of this article is that any Sturmian coloring is a lifting of a coloring of a graph $X$, which is an infinite geodesic or a geodesic ray with loops possibly attached. We further give a complete characterization of $X$ for eventually periodic Sturmian colorings:

\begin{thmx} Let $\phi$ be a Sturmian coloring of a regular tree $T$. 
\begin{enumerate}
\item
There exists a group $\Gamma$ acting on $T$ such that $\phi$ is $\G$-invariant, so that $\phi$ is a lifting of a coloring $\phi_X$ on the quotient graph $X =\Gamma \backslash T$. The quotient graph $X = G \backslash T$ is one of the following two types of graphs. Here, loops are expressed by dotted lines to indicate that they may exist or not.

\medskip


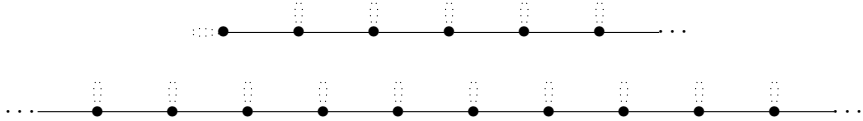
\begin{center}
\begin{tikzpicture}[every loop/.style={}]
  \tikzstyle{every node}=[inner sep=-1pt]
  \node (5) at (-1,0) {$\bullet$};
  \node (6) at  (0,0) {$\bullet$};
  \node (7) at  (1,0) {$\bullet$};
  \node (8) at  (2,0) {$\bullet$};
  \node (9) at  (3,0) {$\bullet$};
  \node (10) at (4,0) {$\bullet$};
  \node (11) at (5,0) {$\cdots$};

\tikzstyle{every loop}=   [-, shorten >=.5pt]

  \path[-] 
	(5)  edge  (6)
	(6)  edge (7)
	(7)  edge (8)
	(8)  edge (9)
	(9)  edge  (10)
	(10) edge (11);
  \path[dotted] 
		 (5) edge [loop left] (5)
		 (6) edge [loop above] (6)
		 (7) edge [loop above] (7)
		 (8) edge [loop above] (8)
		 (9) edge [loop above] (9)
		 (10) edge [loop above] (10);  
\end{tikzpicture}
\end{center}

\smallskip

\begin{center}
\begin{tikzpicture}[every loop/.style={}]
  \tikzstyle{every node}=[inner sep=-1pt]
  \node (0) at (-6,0) {$\cdots$};
  \node (1) at (-5,0) {$\bullet$} ;
  \node (2) at (-4,0) {$\bullet$} ;
  \node (3) at (-3,0) {$\bullet$};
  \node (4) at (-2,0) {$\bullet$};
  \node (5) at (-1,0) {$\bullet$};
  \node (6) at  (0,0) {$\bullet$};
  \node (7) at  (1,0) {$\bullet$};
  \node (8) at  (2,0) {$\bullet$};
  \node (9) at  (3,0) {$\bullet$};
  \node (10) at (4,0) {$\bullet$};
  \node (11) at (5,0) {$\cdots$};
  
  \tikzstyle{every loop}=   [-, shorten >=.5pt]

  \path[-] 
		 (0)  edge (1)
		 (1)  edge  (2)
		 (2)  edge (3)
		 (3)  edge  (4)
		 (4)  edge (5)
		 (5)  edge (6)
		 (6)  edge  (7)
		 (7)  edge  (8)
		 (8)  edge  (9)
		 (9)  edge  (10)
		 (10) edge (11) ;
  \path[dotted]
		 (1) edge [loop above] (1)
		 (2) edge [loop above] (2)
		 (3) edge [loop above] (3)
		 (4) edge [loop above] (4)
		 (5) edge [loop above] (5)
		 (6) edge [loop above] (6)
		 (7) edge [loop above] (7)
		 (8) edge [loop above] (8)
		 (9) edge [loop above] (9) 
		 (10) edge [loop above] (10) ; 
\end{tikzpicture}
\end{center}

\item If $\phi$ is of bounded type,
then it falls into the first case above, i.e. $\phi$ is a lifting of a coloring of a geodesic ray with loops possibly attached.
\item Moreover, $\phi$ is eventually periodic if and only if $X$ is one of the following two graphs. Here the indices on edges indicate the number of corresponding edges in $T$.

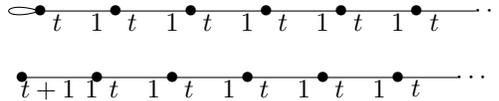
\begin{center}
\begin{tikzpicture}[every loop/.style={}]
  \tikzstyle{every node}=[inner sep=-1pt]

  \node (5) at (-1,0) {$\bullet$};
  \node (6) at  (0,0) {$\bullet$};
  \node (7) at  (1,0) {$\bullet$}; 
  \node (8) at  (2,0) {$\bullet$};
  \node (9) at  (3,0) {$\bullet$};
  \node (10) at (4,0) {$\bullet$};
  \node (11) at (5,0) {$\cdots$};

  \path[-] 
		(5) edge [loop left] 
  	        (5)  edge node [below=2pt] {$t \ \ \ 1$} (6)
		 (6)  edge node [below=2pt] {$t\ \ \ 1$}(7)
		 (7)  edge node [below=2pt] {$t\ \ \ 1$}(8)
		 (8)  edge node [below=2pt] {$t\ \ \ 1$}(9)
		 (9)  edge node [below=2pt] {$t\ \ \ 1$} (10)
		 (10) edge node [below=2pt] {$t\ \ \ $}  (11) ;
\end{tikzpicture}
\end{center}

\smallskip

\begin{center}
\begin{tikzpicture}[every loop/.style={}]
  \tikzstyle{every node}=[inner sep=-1pt]
 
  \node (5) at (-1,0) {$\bullet$};
  \node (6) at  (0,0) {$\bullet$};
  \node (7) at  (1,0) {$\bullet$}; 
  \node (8) at  (2,0) {$\bullet$};
  \node (9) at  (3,0) {$\bullet$};
  \node (10) at (4,0) {$\bullet$};
  \node (11) at (5,0) {$\cdots$};

  \path[-] 
	        (5)  edge node [below=2pt] {$t+1 \ 1$} (6)
		 (6)  edge node [below=2pt] {$t\ \ \ 1$}(7)
		 (7)  edge node [below=2pt] {$t\ \ \ 1$}(8)
		 (8)  edge node [below=2pt] {$t\ \ \ 1$}(9)
		 (9)  edge node [below=2pt] {$t\ \ \ 1$} (10)
		 (10) edge node [below=2pt] {$t\ \ \ $}  (11) ;
\end{tikzpicture}
\end{center}

\end{enumerate}
\end{thmx}

It still remains to characterize the colorings $\phi_X$ on the quotient graph $X$, which we leave for future research.

\section{Subword complexity of colorings of trees}\label{section:2}

Let $T$ be a $k$-regular tree, i.e. a connected graph without loops nor terminal vertices such that the number of edges with a given initial vertex is $k$. Let $G$ be the group of automorphisms of $T$, which is a locally compact topological group with compact-open topology. Let us denote the set of vertices by $VT$ and the set of oriented edges by $ET$. We assume that $ET$ contains $\bar{e}$, which is $e$ with reversed orientation, if it contains $e$. We will denote by $[x,  y]$ the edge from vertex $x$ to vertex $y$.

One nice property of a regular tree is that the automorphisms of the tree are abundant.
An automorphism fixing a vertex $x_0$ is determined by its permutation action on the neighboring $k$ vertices of $x_0$, its action on the 2-sphere from $x_0$ which must be product of permutations on the $k-1$ vertices neighboring each vertex on the 1-sphere, etc. 

In particular, any graph morphism defined from one ball to another ball of same radius extends to an automorphism of $T$ (see \cite{B} for properties of automorphisms fixing a vertex).

\subsection{subword complexity and periodic colorings}\label{subsec:2.1}

Consider the length metric $d$ on $T$ with edge length all equal to $1$. An \textit{$n$-ball around $x$} is defined by $B_n(x) = \{ y \in VT \cup ET : d(x,y) \leq n \}$.
An \textit{$n$-sphere around $x$} is defined by $ \{ y \in VT : d(x,y) = n \}$.

Let us fix a coloring $\phi :VT \to \A$. We say that two balls $B_n(x)$ and $B_n(y)$ are \textit{equivalent} if there exists a color-preserving isomorphism from $B_n(x)$ to $B_n(y)$. By abuse of terminology, we will call such an equivalence class \textit{a colored $n$-ball} or \textit{a coloring of $n$-ball} and denote it by $[B_n(x)]$.

\begin{definition}
Let $\B_\phi(n)$ be the set of colored $n$-balls that appear in $T$ colored by $\phi$. 
The subword complexity $ b_\phi(n)$ of $\phi$ is defined by $b_\phi(n) = | \B_\phi(n)|$. We will denote $b_\phi(n)$ by $b(n)$ if there is no confusion.
\end{definition}

\begin{definition}
We say that two vertices $x$ and $y$ are in the same class if there exists a color-preserving tree automorphism sending $x$ to $y$.
\end{definition}

\begin{lemma}\label{lem:2}
If there exists a sequence $n_k \to \infty$ with $[B_{n_k}(x)] = [B_{n_k}(y)]$, then $x$ and $y$ are in the same class. 
\end{lemma}

\begin{proof}
Let us denote by $f_n$ a color-preserving isomorphism $f_n$ from $B_n(x)$ to $B_n(y)$. 
Since $T$ is a regular tree, each $f_n$ extends to an automorphism of $T$, again denoted by $f_n$, sending $x$ to $y$, although it might not preserve the color outside $B_n(x)$. 

We have a sequence of tree isomorphisms $f_{n_k}$ in the set $\{ f \in G : f(x) = y\}$, which is weakly compact. Thus there exists a subsequence of $f_{n_k}$ which converges weakly to an automorphism $f$ sending $x$ to $y$. This $f$ clearly preserves the color since it preserves the color on arbitrary large balls. Thus $x$ and $y$ are in the same class.
\end{proof}

Let us first examine colorings whose subword complexity $b(n)$ is bounded.

\begin{definition}
A coloring $\phi :VT \to \A$ is periodic if there exists a subgroup $\Gamma \subset G$ such that  $\Gamma \backslash T$ is a finite graph and $\phi$ is $\Gamma$-invariant, i.e.
$$ \phi(\g x) = \phi(x), \text{ for all } x \in VT \text{ and } \g \in \G.$$
Note that we do not require $\Gamma$ to be a discrete subgroup of $G$.
\end{definition}

\vspace{-.05 in}

Let us recall basic notions in the theory of graphs of groups which will be used throughout the paper (see \cite{Serre} and \cite{BK} for details). 

Let $\Gamma$ be a group acting on a $k$-regular tree $T$ by automorphisms. If $\Gamma$ acts without torsion, then the quotient $\Gamma \bs T$ is a $k$-regular graph, but in general, the quotient has a structure of a graph of groups, a graph version of orbifold quotient.  

Let $X$ be a graph, $VX$ its vertex set and $EX$ the set of oriented edges of $X$ ($EX$ contains $e, \bar{e}$ for each unoriented edge in $X$).

A \textit{graph of groups} $(X, G_\bullet )$ is a graph $X$ equipped with a group $G_x$ for each $x \in VX \cup EX$ and an injective homomorphism $j_e : G_e \to G_{\partial_0(e)}$ from the edge group to the group of the initial vertex $\partial_0(e)$ of the edge $e$, for each oriented edge $e\in EX$.

A graph of groups associated to a group $\Gamma$ acting on $T$, which will be denoted by $\Gamma \bs \bs T$, is defined as follows. 
First we assume that $\Gamma$ acts without inversions by taking the first barycentric subdivision of $T$ if necessary.
Take the quotient graph $\Gamma \bs T$ as the underlying graph $X$. 
Choose a connected fundamental domain $D$ of $X$ in $T$, for which the natural projection $D \to X$ is a bijection. Let $\widetilde{x}$ be the lift of $x$ in $D$.

For $x \in VX \cup EX$, set $G_x = Stab_\Gamma (\widetilde x)$. 
If $\partial_0(\widetilde{e})=\widetilde{\partial_0(e)}$, then $G_e \subset G_{\partial_0(e)}$, since an automorphism fixing an edge fixes its initial vertex and terminal vertex. 
Thus we take the inclusion $\iota_e$ for the injective homomorphism $j_e$. 
If $\partial_0(\widetilde{e})\neq \widetilde{\partial_0(e)}$, then there exists an element $\gamma \in \Gamma$ which sends $\partial_0(\widetilde{e})$ to $\widetilde{\partial_0(e)}$ (since their projections are both $\partial_0(e)$ in $\Gamma \bs T$). 
In this case, we take the inclusion composed with conjugation by $\gamma$ as the injective homomorphism $j_e = \gamma \circ \iota_e \circ \gamma^{-1}$.

Let $i : EX \to \mathbb N$ be a map. We call $(X,i)$ an \textit{edge-indexed graph of $T$} if the universal cover of $(X,i)$ is isomorphic to $T$. The universal covering tree of $(X,i)$ is constructed as follows \cite{Bass}. 

Start with a vertex $v_0$ in $X$. For each edge $e$ with initial vertex $v_0$ and index $i(e)$, draw $i(e)$ edges $e_j, j = 1, \cdots, i(e)$ with initial vertex $v_0$, which are liftings of the edge $e$. 
For each terminal vertex $\partial_1(e_j)$ of $e_j$, and for each edge $f$ in $X$ with initial vertex $\partial_0(f) = \partial_1(e)$, draw again $i(f)$ edges $f_j, j =1, \cdots, i(f)$, which are liftings of $f$. Repeat this process to obtain a locally finite tree.

Thus for $k$-regular tree $T$, $(X,i)$ is an edge-index graph of $T$ if and only if for every vertex $x$, the sum of indices $i(e)$ of edges $e$ with initial vertex $x$ equals $k$.

The edge-indexed graph of a graph of groups $(X, G_\bullet)$ is an edge-indexed graph whose graph is the underlying graph $X$ and for which $i(e)$ is the index of $G_e$ in $G_{\partial_0(e)}$. The universal cover of a graph of groups is isomorphic to the universal cover of its edge-indexed graph \cite{Bass}.

Combining these facts, if a coloring $\phi$ is $\Gamma$-invariant, then $\phi$ is determined by a coloring on the edge-indexed graph of $\Gamma \bs \bs T$.
We conclude that $\phi$ is periodic if and only if it is a lift of a coloring on an edge-indexed finite graph.

We will often express our coloring on an edge-indexed graph, of which the underlying graph is infinite unless $\phi$ is periodic.

%
%


The following proposition is an analogue of the classical theorem of Hedlund and Morse \cite{HM}.
\begin{theorem}\label{thm:periodic} 
Let $\phi : VT \to \A$ be a coloring. The followings are equivalent.
\begin{enumerate}
\item The coloring $\phi$ is periodic.
\item The subword complexity of $\phi$ satisfies $b_\phi(n+1) = b_\phi(n)$ for some $n>0$.
\item The subword complexity $b_\phi(n)$ is bounded.
\end{enumerate}
\end{theorem}
\begin{proof}
(1) implies (3)  :  Suppose $\phi$ is periodic i.e. $\Gamma$-invariant for some cocompact subgroup $\Gamma$ of $G$. Let $VX =V( \Gamma \bs \bs T)$ be the vertex set of the quotient graph of groups $\Gamma \bs \bs T$ colored by $\phi$. Then any coloring of $n$-ball in $T$ is determined by the class of its center in $VX$. Thus $b(n) \leq |VX|$. 

\noindent
(3) implies (2) : Since $b_\phi(n)$ is non-decreasing, boundedness of $b_\phi(n)$ implies that $b_\phi(n+1) = b_\phi(n)$ for some $n$.

\noindent
(2) implies (1) :  Suppose $b_\phi(n+1) =b_\phi(n)$. Let us construct an edge-indexed graph $X$ as follows. 

The vertices $VX$ are elements of $\B_\phi(n)$. 
The directed edges $EX$ are all pairs $[[B_n(x)],[B_n(y)]]$ with $d(x,y)=1$.
For a given vertex in $VX$, let us choose a representative $B_n(x)$. 
Condition (2) implies that $[B_n(x)]$ has a unique extension to an $(n+1)$-ball $[B_{n+1}(x)]$. In particular, it implies that for a vertex $y$ of distance 1 from $x$, the color of $B_n(y)$ is uniquely determined up to isomorphism by $[B_n(x)]$. 
Thus the edges defined above are well-defined (i.e. it is independent of the representative of  $[B_n(x)]$).

For any given edge $ e\in EX$, let us choose a representative $[B_n(x), B_n(y)]$ with $d(x,y)=1$.  
Put on each oriented edge $e$ an index $i(e)$ which is the number of vertices $y' \in VT$ with $d(x, y')=1$ and $[B_n (y)]= [B_n(y')]$. This number is independent of the representative of the vertex by the previous paragraph. Thus the edge-indexed graph is well-defined, its underlying graph is finite, and its universal cover is a regular tree. The coloring $\phi$ is clearly determined by its values on $(X,i)$.

Let $\Gamma$ be the group of color-preserving automorphisms of $T$. By the assumption of (2), $[B_n(x)]=[B_n(y)]$ implies that $[B_m(x)]=[B_m(y)]$ for every $m$, thus there is an automorphism of $T$ sending $x$ to $y$ by Lemma~\ref{lem:2}. Conversely, if $[B_n(x)] \neq [B_n(y)]$, then there is no element $\gamma \in \Gamma$ sending $x$ to $y$. Thus there is a bijection between $\Gamma \bs VT$ and $VX$, say $\Psi : \Gamma \bs VT \to VX$, so that the following diagram commutes.

\begin{center}
\begin{tikzpicture}[scale=1.5]
\node (A) at (1,1) {$T$};
\node (B) at (0,0) {$\Gamma \backslash T$};
\node (C) at (2,0) {$X$};
\path[->,font=\scriptsize]
(A) edge (B)
(A) edge (C)
(B) edge (C);
\end{tikzpicture}
\end{center}


Therefore, edges in $\Gamma \bs T$ correspond to edges in $X$. It remains to show that if there is an edge of index $j$ in $X$, i.e. if there exists $x$ and $y_1, \cdots, y_{l_1}$ of the same class in the 1-neighborhood of $x$ , then $[\Gamma_x : \Gamma_e] =l_1$.
It comes from the property of a regular tree that the automorphism group is very large \cite{B}. Let $e$ be an edge with initial vertex $x$ and terminal vertex $v_{11}$. Let $v_{12}, \cdots, v_{1l_1}$ be vertices in the same class as $v_{11}$.
Let the 1-neighborhood of $x$ be partitioned into $j$ sets $\{v_{11}, \cdots, v_{1l_1}\}, \{ v_{21}, \cdots v_{2l_2}\}, \cdots, \{v_{j1}, \cdots, v_{jl_j} \}$  of vertices of the same type. Then 
$$ [\Gamma_x : \Gamma_e] = \frac{ l_1! \cdots l_j!}{(l_1-1)! l_2! \cdots l_j!} = l_1.$$

Thus we conclude that the edge-indexed graph of $\Gamma \bs \bs T$ is isomorphic to $(X,i)$,
and $\phi$ is $\Gamma$-invariant. Thus $\phi$ is periodic.
\end{proof}

\begin{example}\label{rmk1} Let $\Gamma = \langle a_1, \cdots, a_k: a_i^2=1 \rangle$ and $T$ its Cayley graph. Then to any element $g$ of $Aut(T)$ is associated a vertex coloring of $T$ as follows. 

For every vertex $t$ in $T$, there exists a unique element $\g_t$ of $\Gamma$ sending the identity to $t$. Then the element $ \g_{g(t)}^{-1} \circ g \circ \g_t$ sends the identity element back to itself, thus it is a stabilizer of the identity. Let $\phi_g(t)$ be the map $ \g_{g(t)}^{-1} \circ g \circ \g_t$ restricted to the 1-sphere of the identity. We obtain an element of $S_k$, where $S_k$ is the symmetric group on the set of vertices of $1$-sphere of identity. Therefore, we may consider $\phi_g:VT \to S_k$ as a coloring with $\A = S_k$. 

Lubotzky, Mozes and Zimmer showed that $\phi_g$ is a periodic coloring if and only if $g$ is an element of the commensurator group of $\Gamma$ \cite{LMZ}. 

More generally, if $T$ is a locally finite tree, $G=Aut(T)$ is its automorphism group, and $\G$ is a cocompact discrete subgroup of $G$, then to any automorphism is associated a coloring $\phi_g : T \to Y=\Gamma \backslash T$, which is a covering map. An automorphism $g$ is in the commensurator group of $\G$ if and only if its associated coloring $\phi_g$ is periodic \cite{ALN}.
\end{example}

\begin{corollary}\label{coro:comm} With $T$ and $\Gamma$ as in Example~\ref{rmk1},
an automorphism $g$ of $T$ is contained in the commensurator subgroup of $\Gamma$ if and only if its subword complexity $b_{\phi_g}(n)$ is bounded.
\end{corollary}

\subsection{Type set of vertices, colorings of bounded type and eventually periodic colorings}

For the rest of this article, we will study colorings of unbounded subword complexity. For such colorings, special balls play an important role. Before defining special balls, let us provide a basic lemma about graph isomorphisms of balls and branches of a regular tree.

Let $x_i$, $y_i$, $i=1, \cdots k $ be the neighboring vertices of $x$, $y$, respectively. 
Let $\BR_{n} (x,x_i)$ be the \textit{$n$-branch from $x$ to $x_i$} which is defined as the subtree of $B_n(x)$ with vertex set $\{ x\} \cup \{ y \in B_{n}(x)\; | \; d(y,x_i) < d(y,x) \}$ and the edge set given by the set of all edges with initial and terminal vertices in the vertex set, so that 
$$ B_n(x) = \bigcup_{i=1}^k \BR_{n}(x,x_i).$$
Note also that $\BR_{n+1}(x_i, x) = \underset{j\neq i}{\cup} \BR_{n}(x,x_j) \cup [x_i, x]$.

If $f$ is an isomorphism from $B_n(x)$ to $B_n(y)$, then
$f(x)= f(y)$ and $f (\BR_{n} (x,x_i)) = \BR_{n} (y,y_{\sigma(i)})$ for all $i$ for some permutation $\sigma$ on $k$ letters.
Thus we obtain the following lemma.
\begin{lemma}\label{lem:1}
Let us denote the equivalence class up to color-preserving graph isomorphisms of $\BR_n(x,x')$ fixing $x$ by $[\BR_n(x,x')]$ as we did for balls.

\begin{enumerate}
\item If $[B_n(x)]=[B_n(y)]$, then $[\BR_n(x,x_i)]=[\BR_n(y, y_{j})]$ and $[\BR_{n+1}(x_i, x)] = [\BR_{n+1}(y_{j},y)]$ for some $j$.
\item If $[\BR_n(x,x_i)] = [\BR_n(y,y_{\sigma(i)})]$ for all $i=1, \cdots k$ and for some permutation $\sigma \in S_{k}$, then $[B_n(x)]=[B_n(y)]$. 
\item If either $[\BR_n(x,x_i)] \neq [\BR_n(y, y_{j})]$ and $[\BR_{n+1}(x_i, x)] = [\BR_{n+1}(y_{j},y)]$ or $[\BR_n(x,x_i)]=[\BR_n(y, y_{j})]$ and $[\BR_{n+1}(x_i, x)] \neq [\BR_{n+1}(y_{j},y)]$ for some $j$, then $[B_n(x)]\neq [B_n(y)].$ 
\end{enumerate}
\end{lemma}

\begin{definition}
A colored $n$-ball $[B]$ is \textit{special} if there are two distinct colored $(n+1)$-balls $[B_{n+1}(x)]$ and $[B_{n+1}(y)]$ such that $[B_{n}(x)] = [B_{n}(y)] = [B]$.
\end{definition}
If a colored $n$-ball $[B]$ is not special, then it has a unique extension to $(n+1)$-ball in the sense that there is a unique colored $(n+1)$-ball $[B_{n+1}(y)]$ such that $[B_n(y)]=[B]$.

\begin{lemma}\label{new}
If $[B_n(x)] = [B_n(y)]$ and $[B_{n+1}(x)] \neq [B_{n+1}(y)]$, then
for each $1 \le m \le n$ there exist $x'$ from $m$-sphere of $x$ and $y'$ from $m$-sphere of $y$ such that 
$$[B_{n-m+1}(x')] \neq [B_{n-m+1}(y')], \quad [B_{n-m}(x')] = [B_{n-m}(y')]. $$
Consequently, any special $n$-ball contains a special $l$-ball, $\forall l < n$.
\end{lemma}



\begin{proof}
Let $f$ be a color-preserving isomorphism from $B_n(x)$ to $B_n(y)$.
Let $y_i$ be the image of $x_i$ under $f$, for $i=1, \cdots, k$.

By Lemma~\ref{lem:1} (1), $[\BR_{n}(x, x_i)] = [\BR_{n}(y, y_i)]$ and $[\BR_{n+1}(x_i, x)] = [\BR_{n+1}(y_i, y)]$ for all $i$.
Since $[B_{n+1}(x) ] \neq [B_{n+1}(y)]$, it follows that $[\BR_{n+1}(x,x_i)] \neq [\BR_{n+1}(y, y_i)]$ for some $i$ by Lemma~\ref{lem:1} (2). Since $[\BR_n(x_i, x)]=[\BR_n(y_i,y)]$, we get $[B_n(x_i)] \neq [B_n(y_i)]$ by Lemma~\ref{lem:1} (3). We have $[B_{n-1}(x_i)] = [B_{n-1}(y_i)]$ since $[B_n(x)] = [B_n(y)]$. This complete the proof for $m=1$.

Inductively, let $(x_i)_j$ be the vertices neighboring $x_i$ and let $(y_i)_j= f((x_i)_j)$. Since $[x] = [y]$, $[x_i] = [y_i]$ and $[\BR_{n+1}(x,x_i)] \neq [\BR_{n+1}(y, y_i)]$, we have $[\BR_{n}(x_i, (x_i)_j)] \neq [\BR_{n}(y_i, (y_i)_j)]$ for some vertex $(x_i)_j$ neighboring $x_i$ other than $x$.

Note that $\BR_{n-1}((x_i)_j, x_i) \subset B_{n-1}(x)$, thus $[\BR_{n-1}((x_i)_j, x_i)]=[\BR_{n-1}((y_i)_j, y_i)]$.  By Lemma~\ref{lem:1} (3), we get $[B_{n-1} ((x_i)_j)] \neq [B_{n-1} ((y_i)_j)]$. We have $[B_{n-2} ((x_i)_j)] = [f(B_{n-2} ((x_i)_j)]= [B_{n-2} ((y_i)_j)]$.
We repeat this procedure until $m = n$.
\end{proof}

\begin{definition}
The \textit{type set} $\Lambda_x$ of a vertex $x \in VT$ is the set of nonnegative integers $n$ for which $[B_n(x)]$ is special.

A vertex $x$ is said to be \textit{of bounded type} if $\Lambda(x)$ is a finite set. For a vertex $x$ of bounded type, let us denote by $\tau(x)$ the maximum of elements in $\Lambda(x)$ and call it \textit{the maximal type of $x$}. If $\Lambda(x)$ is empty, set $\tau(x) =-1$.
\end{definition}
We will often use the following lemma.
\begin{lemma}\label{lem:basic} 
Let $x$ be a vertex of bounded type.
 We have $\tau(x) \leq m$ if and only if $[B_{m+1}(x)] = [B_{m+1}(y)]$ implies that $x$ and $y$ are in the same class.
\end{lemma}

\begin{proof}
Let $\tau(x) \leq m$ and suppose $[B_{m+1}(x)] = [B_{m+1}(y)]$. If $x$ and $y$ are not in the same class, then $[B_{m+s}(x)] = [B_{m+s}(y)]$ and $[B_{m+s+1}(x)] \neq [B_{m+s+1}(y)]$ for some $s \geq 1$, which implies that $m+s > \tau(x) $ is in the type set, a contradiction. 

Conversely, if $m < t= \tau(x)$, then there exists $y$ such that $[B_t(x)] = [B_t(y)]$ and $[B_{t+1}(x) ] \neq [B_{t+1}(y)]$, thus $[B_{m+1}(x) ]=[B_{m+1}(y)]$ but $x,y$ are not in the same class.
\end{proof}

\begin{definition}
A coloring $\phi$ is called to be \textit{of bounded type} if it has a vertex of bounded type. 
\end{definition}
In fact, if a vertex is of bounded type, then every vertex is of bounded type by the following lemma.

\begin{lemma}\label{lem:bounded}
If a coloring $\phi$ on $T$ is of bounded type, then every vertex of $T$ is of bounded type.
\end{lemma}

\begin{proof}
Let $x$ be a vertex of bounded type. It suffice to show that if $x'$ is a vertex of distance $1$ from $x$, then $x'$ is also of bounded type.

Let $\tau(x)=m$. Let $x_i, i=1, \dots, k$ be the neighboring vertices of $x$.
If $x_i$ and $x_j$ are not in the same class, by Lemma~\ref{lem:2}, there is an integer $n$ such that $[B_{n} (x_i)] \ne [B_{n} (x_j)]$. Let $n_{ij}$ be the minimum of such $n$'s if it exists.
Denote by $N$ the maximum of $m+1$ and $n_{ij}$'s (for $i,j$ such that $x_i$, $x_j$ are not in the same class). 

We claim that the elements of the type set of every $x_i$ are bounded by $N$.
Indeed, suppose that for some $l>N$ and some vertex $z$, $[B_{\ell} (x_i)] = [B_{\ell} (z)]$, say by a color-preserving graph isomorphism $f : B_{\ell} (x_i) \to B_{\ell} (z)$.
Then, since $B_{m+1} (x) \subset B_{\ell} (x_i)$, we have
$[B_{m+1} (x)] = [B_{m+1} (f(x))]$. 

Since $\tau(x) = m$, $x$ and $f(x)$ are in the same class by Lemma~\ref{lem:basic}.
Let $g$ be a color-preserving tree automorphism sending $f(x)$ to $x$.
Then $d(g (z), x) = d(z, f(x))= 1$ which implies that 
$g (z) = x_j$ for some $j$, i.e. $z$ and $x_j$ are in the same class.

Hence, $g \circ f$ is a color-preserving graph isomorphism from $B_{\ell} (x_i)$ to $B_{\ell} (x_j)$,
i.e., $[B_{\ell} (x_i) ] = [B_{\ell} (x_j)]$, 
which is followed by $[B_{N} (x_i)] = [B_{N} (x_j)]$ since $\ell > N$.
If $x_i$ and $x_j$ are not in the same class, then since $N \geq n_{ij}$, $[B_N(x_i)] \neq [B_N(x_j)]$, which is a contradiction. Thus, $x_i$ and $x_j$ are in the same class. Since $x_j$ and $z$ are in the same class, it follows that $x_i$ and $z$ are in the same class. By Lemma~\ref{lem:basic}, we have $\tau(x_i) \leq N$.
\end{proof}


\begin{lemma}\label{lem:type2}
If there exists a special $\ell$-ball, then for every vertex $x$, there exists $M$ depending on $[B_\ell(x)]$ such that $B_{\ell+M}(x)$ contains a special $\ell$-ball.
Moreover, if the alphabet is finite, then there exists a constant $N$ such that every $N$-ball in the tree contains a special $\ell$-ball.  
\end{lemma}

\begin{proof} 
It depends only on $[B_\ell(x)]$ whether $[B_\ell(x)]$ is special or not. If it is not special, there exists a unique extension $[B_{\ell+1}(x)]$. It depends only on $[B_{\ell+1}(x)]$ (thus depends only on $[B_\ell(x)]$) whether $[B_{\ell+1}(x)]$ is special or not. If it is special, then one of the vertices neighboring $x$ is the center of the special $\ell$-ball by Lemma~\ref{new}. If $[B_{\ell+1}(x)]$ is not special, then there exists a unique extension $[B_{\ell+2}(x)]$. It depends only on $[B_{\ell+2}(x)]$ (thus depends only on $[B_\ell(x)]$) whether $[B_{\ell+2}(x)]$ is special or not.

Repeat this process to prove the first part of the lemma. Note that the process ends in a finitely many steps, otherwise, $[B_{\ell+m}(x)]$ does not contain any special $\ell$-ball for any $m$, which implies that $T$ does not have a special $\ell$-ball, a contradiction.

Denote $M$ in the first part of the lemma by $M([B_{\ell}(x)])$. If the alphabet is a finite set, then there are finitely many isomorphism classes of $\ell$-balls. Thus, $N = \ell + \max \{M([B_\ell(x)])\}$ is finite. We conclude that there exists $N>0$ such that for every vertex $x$, $B_{N}(x)$ includes a special $\ell$-ball.
\end{proof}

For the rest of this section, let us study eventually periodic colorings, before we study Sturmian colorings in the next section.

Let us fix a coloring $\phi$. Let $K$ be a finite subset of $T$. 
A coloring on a subtree $U$ has a periodic extension if there exists a periodic coloring $\bar \phi$ on $T$ such that $\bar \phi |_{U} = \phi$. 

\begin{definition}\label{def:event}
A coloring $\phi: VT \to \A$ is called \textit{eventually periodic} if there exists a subtree $K$ of finite number of vertices 
such that $T-K =\bigcup T_i $ is a finite union of subtrees $T_i$ such that $\phi$ on each $T_i$ has a periodic extension $\phi_i$. 
\end{definition}

One may assume $K$ to be a finite ball by taking a ball containing a subtree $K$. A periodic coloring is clearly eventually periodic.

For connected sets $K, K'$, let us denote $[K]=[K']$ if there exists a color-preserving graph isomorphism between them.

\begin{lemma}\label{lemma:2}
If a non-periodic coloring $\phi$ is eventually periodic, then there exists a finite colored subtree which appears exactly once. In fact, we may choose $K$ in Definition~\ref{def:event} to be such a subtree.
\end{lemma}

\begin{proof} 
Choose a ball $K$ satisfying Definition~\ref{def:event} and let $K_r = \{ x \in T : d(x,K) \le r \}$.
We first claim that there exists some $r$ such that $[K_r]$ appears only finitely many times.

Suppose the colored ball $[K_r]$ appears in $T$ infinitely many times.
Since $K_r$ is connected and $T_i$ are connected components of $T-K$, there exists some $T_{i}$ in which $[K_r]$ appears infinitely many times.
Since $(K_r)_{r \in \mathbb{N}}$ is an increasing sequence of balls, if $[K_{r+1}]$ appears in $T_i$, then so does $[K_r]$. Thus there is a subtree $T_i$ in which $[K_r]$ appears infinitely many times for all $r > 0$.
For each $n$-ball $B$ of $T$, $B \subset K_r$ for $r$ large enough.
Therefore, we have $b_\phi (n) \le b_{\phi_i} (n)$, which is bounded, which contradicts the non-periodicity of $\phi$.

Let $\bar K$ be the minimal connected set containing all the colored balls equivalent to $K_r$ which appears only finitely many times, say $N$ times.
Such a subtree is unique since there is a unique path between given balls.
If $K' $ satisfies $[\bar K] = [K']$,
then $K'$ also contains $N$ colored balls equivalent to $K_r$, thus it contains all colored balls equivalent to $K_r$. Thus $\bar K \cap K'$ contains all colored balls equivalent to $[K_r]$.

The minimality condition implies that $\bar K = \bar K \cap K'$, which implies $\bar K = K'$.
Hence, the coloring of $\bar K$ appears in $T$ exactly once.
\end{proof}

\begin{proposition}
Any eventually periodic coloring $\phi$ is of bounded type. 
\end{proposition}

\begin{proof}
By the definition there exists a finite set $K$ 
such that $T-K = \bigcup T_i$ is a finite union of subtrees $T_i$, each of which $\phi$ has a periodic extension $\phi_i :T \to \A$ on.
If $\phi$ is eventually periodic but not periodic, then by Lemma~\ref{lemma:2}, 
we may assume that the colored set $[K]$ appears in $T$ exactly once.

Let $p_i$ be an integer that satisfies $b_{\phi_i} (p_i+1) = b_{\phi_i} (p_i)$ and let $p = \max \{ p_i \} \ge 0$.
Let $x$ be a vertex in $K$. It is enough to show that the vertex $x$ is of bounded type.
We claim that $[B_{n} (x)]$ is not special for $n \ge d+ 4p+4$,
where $d$ is the diameter of $K$. 

Indeed, suppose that $x'$ is a vertex in $T$ with $[B_{n}(x)] = [B_{n}(x')]$ and $[B_{n+1}(x)] \neq [B_{n+1}(x')]$.
By Lemma~\ref{new}, there exist a vertex $y$ on the $(n-2p-1)$-sphere of $x$ and a vertex $y'$ on the $(n-2p-1)$-sphere of $x'$ such that 
$$[B_{2p+2}(y)] \neq [B_{2p+2}(y')] \text{ and } [B_{2p+1}(y)] = [B_{2p+1}(y')].$$
Let us denote by $f : B_n(x) \to B_n(x')$ a color-preserving isomorphism.
Then $f(K) = K$ since $[K]$ appears only once. Note that $d(K,y') > 2p+2$, $d(K, y) > 2p +2$. 
It follows that $B_{2p+2}(y'), B_{2p+2}(y) \subset T - K =\bigcup T_i$, say
$B_{2p+2}(y) \subset T_1$, $B_{2p+2}(y') \subset T_2$.
Thus colored $p$-balls contained in $B_{2p+2}(y)$ and $B_{2p+2}(y')$ are colored $p$-balls contained in $T_1$, $T_2$, respectively.

Let $X_i = \G_i \bs \bs T$ where $\G_i$ be the group of automorphisms of $T$ leaving $\phi_i$ invariant.
Since $|VX_1 |, |VX_2 | \le p+2$, $X_1, X_2$ have diameter at most $p+2$. It follows that $B_{2p+1}(y)$ and $B_{2p+1}(y')$ contain all the colored $p$-balls of $T_1$, $T_2$, respectively.
Thus, $[B_{2p+1}(y)] = [B_{2p+1}(y')]$ implies that $\phi_1$ and $\phi_2$ are isomorphic.
We conclude that $[B_{2p+2}(y)] = [B_{2p+2}(y')]$, which is a contradiction.

If $\phi$ is periodic, then $\phi$ is clearly of bounded type.
\end{proof}

\begin{remark} We allow $\Gamma$ to act with inversions (i.e. without automorphisms fixing an edge and exchanging initial and terminal vertices). The resulting graph of groups will be the usual graph of groups (for groups acting without inversions) of the first barycentric subdivision $T'$ of $T$. 

For edges with initial and terminal vertices of different classes, there is no color-preserving inversion for them.
If there exists an edge with initial and terminal vertices of the same class, then there exists a color-preserving inversion exchanging initial and terminal vertices, which results in an edge in $T'$, corresponding to a half-edge in $T$. This edge will be drawn as a half-edge in the quotient graph, with one white vertex and one black vertex. We will call it a \textit{loop} since it is an edge whose initial vertex is its terminal vertex.

The black vertices in the next figure are the vertices of $T$, whereas the white vertex at the left end of the ray is a vertex in $T'-T$, which was added during the barycentric subdivision. We will omit other white vertices if the indices of edges around it are all 1.
\end{remark}

Let us give an example of non-eventually periodic coloring of bounded type. We will study such colorings with minimal subword complexity in Section~\ref{subsec:sturmianbounded}.

\begin{example} (A non-eventually periodic coloring of bounded type)

Consider a coloring of an edge-indexed graph given as follows:

\medskip
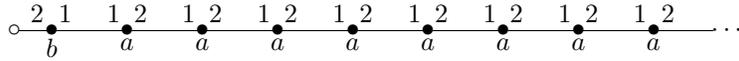
\begin{center}
\begin{tikzpicture}[every loop/.style={}]
  \tikzstyle{every node}=[inner sep=-1pt]
  \node (-1) at (-.5,0) {$\circ$};
  \node (0) {$\bullet$} node [below=4pt] {$b$};
  \node (1) at (1,0) {$\bullet$} node [below=4pt] at (1,0) {$a$};
  \node (2) at (2,0) {$\bullet$} node [below=4pt] at (2,0) {$a$};
  \node (3) at (3,0) {$\bullet$} node [below=4pt] at (3,0) {$a$};
  \node (4) at (4,0) {$\bullet$} node [below=4pt] at (4,0) {$a$};
  \node (5) at (5,0) {$\bullet$} node [below=4pt] at (5,0) {$a$};
  \node (6) at (6,0) {$\bullet$} node [below=4pt] at (6,0) {$a$};
  \node (7) at (7,0) {$\bullet$} node [below=4pt] at (7,0) {$a$};
  \node (8) at (8,0) {$\bullet$} node [below=4pt] at (8,0) {$a$};
  \node (9) at (9,0) {$\cdots$};

  \path[-] (-1) edge node [above=4pt] {\ 2} (0)    
		 (0) edge node [above=4pt] {1 \quad 1} (1)
		 (1) edge node [above=4pt] {2 \quad 1} (2)
		 (2) edge node [above=4pt] {2 \quad 1} (3)
		 (3) edge node [above=4pt] {2 \quad 1} (4)
		 (4) edge node [above=4pt] {2 \quad 1} (5)
		 (5) edge node [above=4pt] {2 \quad 1} (6)
		 (6) edge node [above=4pt] {2 \quad 1} (7)
		 (7) edge node [above=4pt] {2 \quad 1} (8)
		 (8) edge node [above=4pt] {2 \quad \ } (9);
\end{tikzpicture}
\end{center}
Since the vertex colored by $b$ has the empty set as its type, it is of bounded type.
The universal covering tree $T$ has all vertices colored by $a$ except one geodesic $\mathbf{p}$ which is colored by $b$'s.

Admissible colored balls are the ones with vertices all colored by $a$ except vertices on $\mathbf{p}$ if there is a non-trivial intersection. These colored $n$-balls are determined (up to automorphisms of colored balls) by the distance $d$ ($d\in \{0, \cdots n\}$) of $\mathbf{p}$ from the center of the ball. For each $d$, there are infinitely many vertices of distance $d$ from $\mathbf{p}$, thus every class of colored $n$-balls appear infinitely many times.  
Thus it is not eventually periodic by Lemma~\ref{lemma:2}.
Remark that
$b(n) = n+2.$

\end{example}

\section{minimal subword complexity and Sturmian colorings}\label{sec3}

In this section, we study Sturmian colorings, i.e. colorings of minimal unbounded subword complexity.
We show the main theorem stated in the introduction.

\begin{definition}
A coloring $\phi$ of a $k$-regular tree $T$ is called Sturmian if  $b_\phi(n)=n+2$.
\end{definition}
 
Sturmian colorings are the colorings with smallest subword complexity among all non-periodic colorings: since $b_\phi(0)=2$, the strictly increasing condition $b_\phi(n+1) > b_\phi(n)$ implies that $b_\phi(n) \geq n+2$.
Note also that from $b_\phi(0)=2$, the coloring $\phi$ is on two letters.
 
Since $b(n+1) =b(n)+1$, there is exactly one coloring of $k$-ball with two possible extensions to colorings of $(k+1)$-balls.
Therefore, for each $n \ge 0$, there exists a unique special $n$-ball.

\subsection{Sturmian colorings of bounded type}\label{subsec:sturmianbounded}

Let $\phi$ be a coloring of bounded type. For general such $\phi$,
there might be two vertices, with the same maximal type, which are centers of distinct colored $\ell$-balls for some $\ell$.
However, if we assume that $b(n) = n+2$, then there are no such vertices, i.e. 
the colored balls around each vertex $v$ are completely determined by the maximal type of $v$:

\begin{proposition}\label{maximal_type}
For a Sturmian coloring,
if two vertices $x$ and $y$ have the same maximal type, then $x$ and $y$ are in the same class. 
\end{proposition}

\begin{proof}
Suppose that two vertices $x$ and $y$, not in the same class, have the same maximal type $\ell$.
Then by the uniqueness of the $\ell$-special ball we have 
$[B_\ell(x)] = [B_\ell(y)]$. By Lemma~\ref{lem:basic}, we have $[B_{\ell+1}(x)] \ne [B_{\ell+1}(y)]$.
 
By Lemma~\ref{lem:type2}, for every vertex $w$, there exists $N>0$ such that $B_{\ell+N}(w)$ includes the special $\ell$-ball, of center say $z$, i.e., $[B_\ell(z)]=[B_\ell(x)]$. Thus $[B_{\ell+1}(z)]$ is either $[B_{\ell+1}(x)]$ or $[B_{\ell+1}(y)]$ since there are only two possible extensions of $[B_\ell(x)]$ to colored $(\ell+1)$-balls. 
Therefore, by Lemma~\ref{lem:basic}, $z$ is in the same class as either $x$ or $y$.
Since $w$ is arbitrary, the whole tree is covered by $N$-balls of centers which are in the same class with either $x$ or $y$.

Let $L$ be the maximum of the maximal types of the vertices in $B_N (x)$ and $B_N (y)$.
By Lemma~\ref{lem:bounded},  $L$ is finite.
By definition, the maximal type of any vertex $w$ is bounded by $L$,
which leads to a contradiction to the fact that there exists a vertex of arbitrary large maximal type since special $m$-ball exists for every $m$. 
\end{proof}

It follows from Proposition~\ref{maximal_type} that the projection $\pi : VT \to VX=VT/\sim$, where $v \sim w$ if they have the same maximal type, is well-defined and extends to a graph morphism $T \to X=T/\sim$, again denoted by $\pi$.
We have a quotient graph $X$ and a coloring $\phi_X$ on $X$ such that
$\phi = \phi_X \circ \pi$.
Note that the vertices of $X$ are determined by their maximal type.
The following lemma shows what the admissible edges of $X$ are, in terms of their maximal type:

\begin{lemma}\label{edge}
In a Sturmian coloring, if a vertex $v$ is of maximal type $m$, then 
\begin{enumerate}
\item its neighboring vertices are of maximal type $m-1$, $m$ or $m+1$,
\item one of its neighboring vertex is of maximal type $m+1$.
\item if $m$ is not minimum among maximal types of vertices, one of its neighboring vertex is of maximal type $m-1$.
\end{enumerate}
\end{lemma}

\begin{proof}
(1) Let $\tau(x) = m$. We claim that its neighboring vertices are of maximal type at most $m+1$. By claim, if a neighboring vertex is of maximal type $l < m-1$, then $\tau(x) \leq l+1 < m$, which is a contradiction.

Now let us prove the claim.
Suppose $\tau(x_i) = \tau > m+1$. 
Choose $x_i$ such that $\tau(x_i) = \max \{\tau(x_j)\}$. 
Since $\tau(x_i) = \tau$, there exists $y$ such that $[B_\tau(x_i)] =[B_\tau(y)]$ and $[B_{\tau+1}(x_i)] \neq [B_{\tau+1}(y)]$. Let $f : B_\tau (x_i) \to B_\tau(y)$ be a color-preserving isomorphism. Then $[B_{m+1}(x)] =[B_{m+1}(f(x))]$ since $m+1 < \tau$. 
By Lemma~\ref{lem:basic}, $x$ and $f(x)$ are in the same class since $\tau(x) = m$. Since $d(y, f(x))=1$, there exists some $j$ such that $y$, $x_j$ are in the same class. Therefore, $[B_\tau(x_i) ]=[ B_\tau(y)]=[B_\tau(x_j)]$ but $[B_{\tau+1}(x_i)]  \neq [ B_{\tau+1}(y)] = [B_{\tau+1}(x_j)]$. Hence, $x_i$ and $x_j$ are not in the same class and $\tau(x_j) \ge \tau = \tau(x_i)$. 
Since $\tau(x_i) = \max \{\tau(x_j)\}$, $\tau(x_j) = \tau(x_i)$, which contradicts Proposition~\ref{maximal_type}.

(2) Since $b(n)$ is not bounded, there is a vertex $w$ whose maximal type is larger than $m$.
Since the maximal type of neighboring vertices may differ by at most one by part (1), there is a vertex $w'$ in the path between $v$ and $w$, of maximal type $(m+1)$, neighboring a vertex $v'$ of maximal type $m$. 
Since $v,v'$ have the same maximal type, by Proposition~\ref{maximal_type}, $v$ and $v'$ are equivalent, thus have the same neighborhoods. Since $v'$ neighbors a vertex of maximal type $m+1$, so does $v$.

(3) If $m$ is not minimal, there exists a vertex $w$ whose maximal type is smaller than $m$. The statement of (3) follows from an argument similar to the previous paragraph.
\end{proof}

Now we have the following theorem by Proposition~\ref{maximal_type} and Lemma~\ref{edge}:

\begin{theorem}\label{prop:5}
If $\phi$ is a Sturmian coloring,
then there exists a proper infinite quotient graph $X$ of $T$ with 
$$VX =\{ m,m+1, m+2, \dots , \}, \quad  EX \subset \{ [i,i+1], [i+1, i] \, | \, i \ge m \} \cup \{ [i,i] \, | \, i \ge m \} $$ 
and a coloring $\phi_X$ on $X$ such that
$\phi = \phi_X \circ \pi$, where $\pi: T \to X$ is the canonical quotient map
and $m = \min \{ \tau(x) : x \in VT\}$.
\end{theorem}

\begin{example}
Consider the coloring of edge-indexed graph $X$:

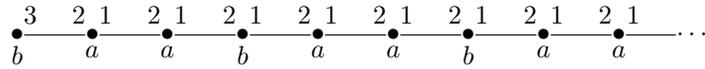
\medskip

\begin{center}
\begin{tikzpicture}[every loop/.style={}]
  \tikzstyle{every node}=[inner sep=0pt]
  \node (0) {$\bullet$} node [below=4pt] {$b$};
  \node (1) at (1,0) {$\bullet$} node [below=4pt] at (1,0) {$a$};
  \node (2) at (2,0) {$\bullet$} node [below=4pt] at (2,0) {$a$};
  \node (3) at (3,0) {$\bullet$} node [below=4pt] at (3,0) {$b$};
  \node (4) at (4,0) {$\bullet$} node [below=4pt] at (4,0) {$a$};
  \node (5) at (5,0) {$\bullet$} node [below=4pt] at (5,0) {$a$};
  \node (6) at (6,0) {$\bullet$} node [below=4pt] at (6,0) {$b$};
  \node (7) at (7,0) {$\bullet$} node [below=4pt] at (7,0) {$a$};
  \node (8) at (8,0) {$\bullet$} node [below=4pt] at (8,0) {$a$};
  \node (9) at (9,0) {$\cdots$};

  \path[-]  (0) edge node [above=4pt] {3 \quad 2} (1)
		 (1) edge node [above=4pt] {1 \quad 2} (2)
		 (2) edge node [above=4pt] {1 \quad 2} (3)
		 (3) edge node [above=4pt] {1 \quad 2} (4)
		 (4) edge node [above=4pt] {1 \quad 2} (5)
		 (5) edge node [above=4pt] {1 \quad 2} (6)
		 (6) edge node [above=4pt] {1 \quad 2} (7)
		 (7) edge node [above=4pt] {1 \quad 2} (8)
		 (8) edge node [above=4pt] {1 \quad \ } (9);
\end{tikzpicture}
\end{center}
Let us denote the vertices of $X$ by $\{0, 1, \cdots \}$.
Denote by $\widetilde x$ a lift of $x$ in $T$.
By the periodicity of the coloring $\phi_0$ on $X$, we have
$[B_n (\widetilde{i})] = [B_n(\widetilde{{i+3}})]$ as long as the balls do not contain 1-neighborhood of $\tilde 0$, i.e. for any $i  \geq n$.
Thus $$ \mathcal B(n) = \{ [B_n (\widetilde{i})] : i = 0, 1, \dots, n+2 \}. $$

To show that this example is Sturmian, we only need to show that for $n \ge 1$,
$$[ B_n (\widetilde{{n-1}})] = [B_n(\widetilde{{n+2}})]. $$
Let $y$ and $z$ be lifts of ${n-1}$ and ${n+2}$, respectively.
For $n \ge 1$, we have $[B_{n-1} (y)] = [B_{n-1} (z)]$ by a graph isomorphism $f$ such that $\pi(f(x)) = \pi(x) + 3$ for $0 \le \pi(x) \le 2n-2$.
Let $x$ be a vertex in $(n-1)$-sphere of $y$.
If $\pi(x) = i$, for $i \ge 1$, then we have $[B_1(x)] = [B_1( f(x) )]$.
If $\pi(x) = 0$, then $[B_1( x )] =  
\includegraphics{Sturmian_tree-12.mps} = [B_1(f(x))] $ since $\pi(f(x)) = 3$.
Hence, by Lemma~\ref{new} we have $ [B_{n} (y)] = [B_{n} (z)]. $

Lift this coloring to the universal covering tree $T$. See Figure~\ref{sturmfig}.
The admissible colored balls are as follows:  
\begin{align*}
\mathcal B(1) &= \Bigg \{ 
 \;\;\;
\includegraphics{Sturmian_tree-12.mps}, \;\;\;
\includegraphics{Sturmian_tree-13.mps}, \;\;\;
\includegraphics{Sturmian_tree-14.mps} \;\;\; 
\Bigg \}, \\
\mathcal B(2) &= \Bigg \{  \;\;\; \includegraphics{Sturmian_tree-15.mps}, \;\;\; \includegraphics{Sturmian_tree-16.mps}  \;\;\;
\includegraphics{Sturmian_tree-17.mps}, \;\;\;
\includegraphics{Sturmian_tree-18.mps}  \;\;\; 
 \Bigg \}.
\end{align*}

\begin{figure}
\begin{center}
\includegraphics{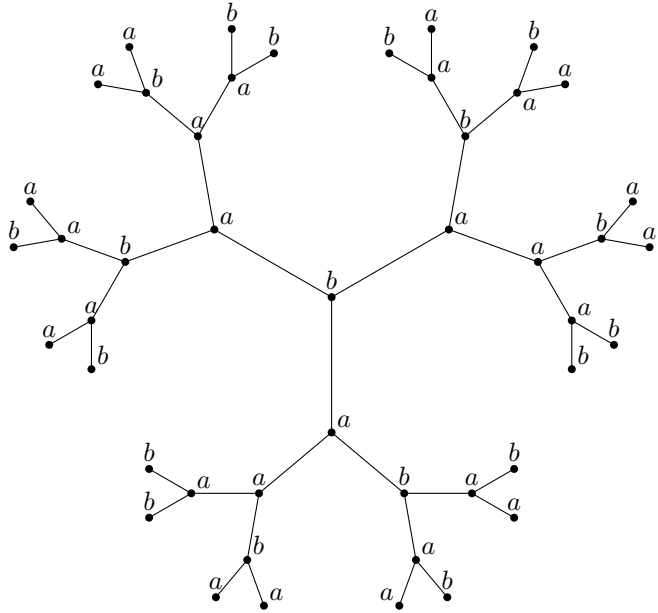}
\end{center}
\caption{An example of Sturmian tree}
\label{sturmfig}
\end{figure}

\end{example}

Here are more examples of Sturmian colorings which are liftings of one-sided periodic colorings on an infinite quotient ray of $T$.

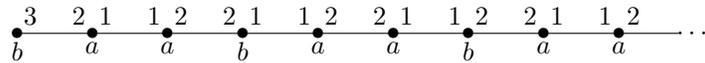
\medskip

\begin{center}
\begin{tikzpicture}[every loop/.style={}]
  \tikzstyle{every node}=[inner sep=-1pt]
  \node (0) {$\bullet$} node [below=4pt] {$b$};
  \node (1) at (1,0) {$\bullet$} node [below=4pt] at (1,0) {$a$};
  \node (2) at (2,0) {$\bullet$} node [below=4pt] at (2,0) {$a$};
  \node (3) at (3,0) {$\bullet$} node [below=4pt] at (3,0) {$b$};
  \node (4) at (4,0) {$\bullet$} node [below=4pt] at (4,0) {$a$};
  \node (5) at (5,0) {$\bullet$} node [below=4pt] at (5,0) {$a$};
  \node (6) at (6,0) {$\bullet$} node [below=4pt] at (6,0) {$b$};
  \node (7) at (7,0) {$\bullet$} node [below=4pt] at (7,0) {$a$};
  \node (8) at (8,0) {$\bullet$} node [below=4pt] at (8,0) {$a$};
  \node (9) at (9,0) {$\cdots$};

  \path[-] (0) edge node [above=4pt] {3 \quad 2} (1)
		 (1) edge node [above=4pt] {1 \quad 1} (2)
		 (2) edge node [above=4pt] {2 \quad 2} (3)
		 (3) edge node [above=4pt] {1 \quad 1} (4)
		 (4) edge node [above=4pt] {2 \quad 2} (5)
		 (5) edge node [above=4pt] {1 \quad 1} (6)
		 (6) edge node [above=4pt] {2 \quad 2} (7)
		 (7) edge node [above=4pt] {1 \quad 1} (8)
		 (8) edge node [above=4pt] {2 \quad \ } (9);
\end{tikzpicture}
\end{center}

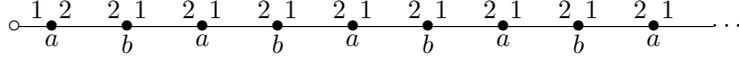
\medskip

\begin{center}
\begin{tikzpicture}[every loop/.style={}]
  \tikzstyle{every node}=[inner sep=-1pt]
  \node (-1) at (-.5,0) {$\circ$};
  \node (0) {$\bullet$} node [below=4pt] {$a$};
  \node (1) at (1,0) {$\bullet$} node [below=4pt] at (1,0) {$b$};
  \node (2) at (2,0) {$\bullet$} node [below=4pt] at (2,0) {$a$};
  \node (3) at (3,0) {$\bullet$} node [below=4pt] at (3,0) {$b$};
  \node (4) at (4,0) {$\bullet$} node [below=4pt] at (4,0) {$a$};
  \node (5) at (5,0) {$\bullet$} node [below=4pt] at (5,0) {$b$};
  \node (6) at (6,0) {$\bullet$} node [below=4pt] at (6,0) {$a$};
  \node (7) at (7,0) {$\bullet$} node [below=4pt] at (7,0) {$b$};
  \node (8) at (8,0) {$\bullet$} node [below=4pt] at (8,0) {$a$};
  \node (9) at (9,0) {$\cdots$};

  \path[-] (-1) edge node [above=4pt] {\ 1} (0)    
		 (0) edge node [above=4pt] {2 \quad 2} (1)
		 (1) edge node [above=4pt] {1 \quad 2} (2)
		 (2) edge node [above=4pt] {1 \quad 2} (3)
		 (3) edge node [above=4pt] {1 \quad 2} (4)
		 (4) edge node [above=4pt] {1 \quad 2} (5)
		 (5) edge node [above=4pt] {1 \quad 2} (6)
		 (6) edge node [above=4pt] {1 \quad 2} (7)
		 (7) edge node [above=4pt] {1 \quad 2} (8)
		 (8) edge node [above=4pt] {1 \quad \ } (9);
\end{tikzpicture}
\end{center}

The following example indicates that the coloring on the quotient graph $X$ can be quite arbitrary.

\begin{example}
Let $X$ be the following infinite edge-indexed colored graph: 

\medskip
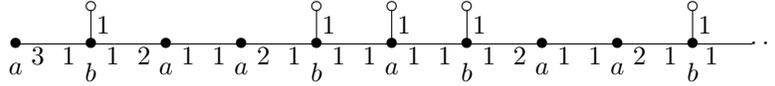
\begin{center}
\begin{tikzpicture}[every loop/.style={}]
  \tikzstyle{every node}=[inner sep=-1pt]
  \node (1) at (-3,0) {$\bullet$} node [below=8pt] at (-3,0) {$a$};
  \node (2) at (-2,0) {$\bullet$} node [below=8pt] at (-2,0) {$b$};
  \node (a2) at (-2,.5) {$\circ$};
  \node (3) at (-1,0) {$\bullet$} node [below=8pt] at (-1,0) {$a$};
  \node (4) at (0,0) {$\bullet$} node [below=8pt] at (0,0) {$a$};
  \node (5) at (1,0) {$\bullet$} node [below=8pt] at (1,0) {$b$};
  \node (a5) at (1,.5) {$\circ$};
  \node (6) at  (2,0) {$\bullet$} node [below=8pt] at (2,0) {$a$};
  \node (a6) at (2,.5) {$\circ$};
  \node (7) at  (3,0) {$\bullet$} node [below=8pt] at (3,0) {$b$};
  \node (a7) at (3,.5) {$\circ$};
  \node (8) at (4,0) {$\bullet$} node [below=8pt] at (4,0) {$a$};
  \node (9) at (5,0) {$\bullet$} node [below=8pt] at (5,0) {$a$};
  \node (10) at (6,0) {$\bullet$} node [below=8pt] at (6,0) {$b$};
  \node (a10) at (6,.5) {$\circ$};
  \node (11) at (7,0) {$\cdots$};

  \path[-] 
		 (1)  edge node [below=2pt] {$3 \ \ 1$} (2)
		 (2)  edge node [below=2pt] {$1 \ \ 2$} (3)
		 (3)  edge node [below=2pt] {$1 \ \ 1$} (4)
		 (4)  edge node [below=2pt] {$2 \ \ 1$} (5)
		 (5)  edge node [below=2pt] {$1 \ \ 1$} (6)
		 (6)  edge node [below=2pt] {$1 \ \ 1$} (7)
		 (7)  edge node [below=2pt] {$1 \ \ 2$} (8)
		 (8) edge node [below=2pt] {$1 \ \ 1$} (9) 
		 (9) edge node [below=2pt] {$2 \ \ 1$} (10) 
		 (10) edge node [below=2pt] {$1 \ \ \ $} (11) 
		 (2)  edge node [right=3pt] {$1$} (a2)    
		 (5)  edge node [right=3pt] {$1$} (a5)    
		 (6)  edge node [right=3pt] {$1$} (a6)    
		 (7)  edge node [right=3pt] {$1$} (a7) 
		 (10)  edge node [right=3pt] {$1$} (a10)    ;  
\end{tikzpicture}
\end{center}

Suppose that $b-a-a-b$ and $b-a-b$ appear arbitrarily in $X$. Then it is not a periodic coloring on $X$.
The minimum of maximal types is $m=0$.  Each vertex in $X$ represents vertices of maximal type $\{ 0 , 1 , 2, \dots \}$, which is the distance of the vertex from the leftmost vertex.

Note that we colored $X$ so that the balls of center colored $a$ or $b$ not intersecting the leftmost vertex is uniquely determined, since all 1-neighborhoods of $a, b$ not intersecting the leftmost vertex are uniquely determined, namely $a$ has neighboring vertices $a,b,b$ and $b$ has neighboring vertices $a,a,b$. 
Therefore, the number of $n$-balls are the balls of center of distance $0, 1, \cdots, n-1$ from the leftmost vertex, plus two balls of center $a$ and $b$, which does not intersect the 1-sphere of the leftmost vertex.
Thus the number of colored $n$-ball is $b(n)=n+2$.
\end{example}

\subsection{Eventually periodic colorings revisited}
In this subsection, we obtain a complete characterization of eventually periodic Sturmian colorings.

\begin{proposition}\label{prop:ep} A Sturmian coloring $\phi$ is eventually periodic if and only if there exists a finite colored subtree which appears exactly once.
\end{proposition}

\begin{proof} It was proved in Lemma~\ref{lemma:2} that an eventually periodic coloring has a finite colored subtree appearing exactly once. Let us prove the other direction.

Let $B_n(x)$ be a ball containing a finite subtree whose coloring appears only once. Then $[B_n(x)]$ appears only finitely many times. Thus $\B(n)$ contains a colored $n$-ball appearing only finitely many times. Let $m+1 \geq 0$ be a minimal such integer.  

By definition, either $m+1 \geq 1$, i.e. there exists a colored $(m+1)$-ball appearing only finitely many times and all colored $m$-balls appear infinitely many times, or $m+1 =0$, i.e. $\B(0)$ contains a letter $b$ appearing only finitely many times. In the latter case, $T - \{y: \phi(y)=b \}$ is colored by one color, thus $\phi$ is eventually periodic. Assume we are in the former case.

Let $[B_{m+1}(y)]$ be a colored $(m+1)$-ball appearing finitely many times. Since $[B_m(y)]$ appear infinitely many times, $[B_m(y)]$ is a special $m$-ball. Since a special $m$-ball is unique and there are exactly two extensions to colored $(m+1)$-balls, there is no colored $(m+1)$-ball other than $[B_{m+1}(y)]$ appearing finitely many times.

Let $K $ be a minimal subtree containing all $(m+1)$-balls equivalent to $B_{m+1}(y)$. 
We claim that for any $z \in V(T -K)$, $[B_{m+1}(z)]$ is uniquely determined by $[B_m(z)]$. Indeed, if $[B_m(z)]$ is not the special $m$-ball, then there is a unique extension of $[B_m(z)]$ to $(m+1)$-ball. If $[B_m(z)]$ is the special $m$-ball, then 
 $[B_m(z)] = [B_m(y)]$ but $[B_{m+1}(z)] \neq [B_{m+1}(y)]$ since all $(m+1)$-balls equivalent to $B_{m+1}(y)$ are in $K$. Since there are only two possible extensions of $[B_m(z)]$ including $[B_{m+1}(y)]$, $[B_{m+1}(z)]$ is uniquely determined.
 
 Now let us construct an edge-indexed graph using the idea in the proof of Theorem~\ref{thm:periodic}. 
Let $VX = \B(m) $ and
$$EX = \{ [[B_m(z)], [B_m(w)]] : z, w \in VT, [B_{m+1}(z)] \ne [B_{m+1}(y)], d(z, w)=1 \}.$$ 
Let us show that $e \in EX$ if and only if $\bar{e} \in EX$. We claim that if $z, w$ are vertices such that $[B_{m+1}(z)] \neq [B_{m+1}(y)]$, $[B_{m+1}(w)] = [B_{m+1}(y)]$ and $d(z,w)=1$, then there exist vertices $z',w'$ such that $d(z',w')=1$, $([B_m(z)], [B_m(w)]) = ( [B_m(z')], [B_m(w')])$  and $[B_{m+1}(w')] \neq [B_{m+1}(y)]$.

Let $z'$ be a vertex of $[B_{m+1}(z')] = [B_{m+1}(z)] \ne [B_{m+1}(y)]$. 
There exist infinitely many such $z'$'s.
For each $z'$, there exists $w'$'s such that $d(z',w')=1$ and $[B_m(w')]=[B_m(w)]=[B_m(y)]$. 
Thus, there exist infinitely many $w'$'s such that $d(z',w')=1$ for some $z'$ and $[B_m(w')]=[B_m(w)]=[B_m(y)]$. 
Since there exist only finitely many $[B_{m+1}(y)]$,
there exist infinitely many $w'$ such that $d(z',w')=1$ for some $z'$ and $[B_{m+1}(w')] \ne [B_{m+1}(y)]$. The claim follows.

Let 
$$i([[B_m(z)], [B_m(w)]]) = |\{ w' \in VT : d(z,w')=1, [B_m(w')]=[B_m(w)] \} |.$$  
Since $[B_{m+1}(z)]$ is uniquely determined by $[B_m(z)]$, as in the proof of Theorem~\ref{thm:periodic}, $X$ is a well-defined edge-indexed graph, possibly disconnected, of which each connected component $(X_j, i)$ has a universal covering tree isometric to $T$. The classes of $m$-balls with center in each given $T_j$ are connected in $(X,i)$, thus they are included in one single connected component of $X$, say $X_j$. Note that $X_k = X_j$ even if $T_k \neq T_j$. 

Define a coloring $\phi_j$ on $X_j$ to be $\phi_j([B_m(w)]) = \phi(w)$. Then there is a graph homomorphism $\pi_j : T_j \to (X_j,i)$ given by $\pi_j : w \mapsto [B_m(w)]$ such that $\phi_j \circ \pi_j = \phi$ on $T_j$. Note that $\pi_j$ is surjective since $T_j$ contains arbitrarily large balls.
 Since $T$ is a tree, by extending locally, one may extend $\pi_i$ to a covering $\widetilde{\pi_j} : T \to (X_j, i)$. Lift the coloring on $(X_j,i)$ to $T$, which is clearly an extension of $\phi_{T_j}$.
\end{proof}

We remark that $X$ in the proof above is in fact connected.
\begin{lemma}\label{lem:9}
If a Sturmian coloring $\phi$ is eventually periodic, then there are finitely many vertices of any given class.
\end{lemma}

\begin{proof}
By Proposition~\ref{maximal_type}, a class of vertices is determined by the maximal type. 
Let  $B_n(x)$ be a ball containing the finite subtree appearing only once given by Proposition~\ref{prop:ep}.
Then $[B_n(x)]$ appears only finitely many times. Let $\tau(x) = m$.
Hence,
$$
\# \{ y \in VT : \tau(y) = m \} \le \# \{ y \in VT : [B_n(y)] = [B_n(x)] \} < \infty.
$$
Thus there are finitely many vertices of maximal type $m$.
Since
$$\{ y \in VT : \tau (y) = m \pm 1 \} \subset  \bigcup_{x': \tau(x')= m} \{ y \in VT : d(x', y) = 1 \}$$
by Lemma~\ref{edge},
$ \{ y \in VT : \tau(y) = m \pm 1 \} $ is finite. Repeat this process.
\end{proof}

\begin{theorem}
A Sturmian coloring $\phi$ is eventually periodic if and only if 
the quotient graph $X$ of $T$ in Theorem~\ref{prop:5} is one of the following graphs.

\medskip
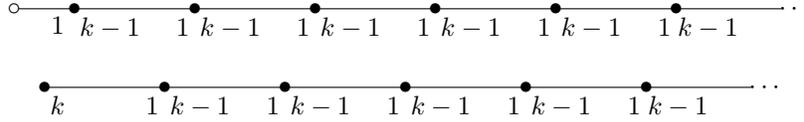
\begin{center}
\begin{tikzpicture}[every loop/.style={}]
  \tikzstyle{every node}=[inner sep=-1pt]
  \node (0) at (-.8,0) {$\circ$};
  \node (1) at (0,0) {$\bullet$};
  \node (2) at (1.6,0) {$\bullet$};
  \node (3) at (3.2,0) {$\bullet$};
  \node (4) at (4.8,0) {$\bullet$};
  \node (5) at (6.4,0) {$\bullet$};
  \node (6) at (8,0) {$\bullet$};
  \node (7) at (9.6,0) {$\cdots$};

  \path[-] 
		 (0)  edge node [below=4pt] {$\quad 1$} (1)
		 (1)  edge node [below=4pt] {$k-1 \ \quad 1$} (2)
		 (2)  edge node [below=4pt] {$k-1 \ \quad 1$} (3)
		 (3)  edge node [below=4pt] {$k-1 \ \quad 1$} (4)
		 (4)  edge node [below=4pt] {$k-1 \ \quad 1$} (5)
		 (5)  edge node [below=4pt] {$k-1 \ \quad 1$} (6)
		 (6)  edge node [below=4pt] {$k-1 \ \quad \ $} (7) ;
\end{tikzpicture}
\end{center}

\medskip
\begin{center}
\begin{tikzpicture}[every loop/.style={}]
  \tikzstyle{every node}=[inner sep=-1pt]
  \node (1) at (0,0) {$\bullet$};
  \node (2) at (1.6,0) {$\bullet$};
  \node (3) at (3.2,0) {$\bullet$};
  \node (4) at (4.8,0) {$\bullet$};
  \node (5) at (6.4,0) {$\bullet$};
  \node (6) at (8,0) {$\bullet$};
  \node (7) at (9.6,0) {$\cdots$};

  \path[-] 
		 (1)  edge node [below=4pt] {$k  \quad \qquad 1$} (2)
		 (2)  edge node [below=4pt] {$k-1 \ \quad 1$} (3)
		 (3)  edge node [below=4pt] {$k-1 \ \quad 1$} (4)
		 (4)  edge node [below=4pt] {$k-1 \ \quad 1$} (5)
		 (5)  edge node [below=4pt] {$k-1 \ \quad 1$} (6)
		 (6)  edge node [below=4pt] {$k-1 \ \quad \ $} (7) ;
\end{tikzpicture}
\end{center}

\end{theorem}
\begin{proof}
Let us denote by $m$ the leftmost vertex in the graph of Theorem~\ref{prop:5}. If there exist two loops at $m$, 
 then there exists an infinite geodesic in $T$ all of whose vertices are lifts of $m$, which is a contradiction to Lemma~\ref{lem:9}.
If the index of the edge $[l, l-1]$ is larger than 1, then there exists a geodesic whose vertices are of maximal type 
$$\cdots, m+1, m, m+1, \cdots, l, l+1, l \cdots, m+1, m, m+1, \cdots, $$
which is again a contradiction.
If there is a loop from vertex $l$ to itself, then there exists a geodesic whose vertices are of maximal type
$$ \cdots, m, m+1, \cdots, l-1, l, l , l-1, \cdots, m+1, m, m+1, \cdots, $$
again a contradiction.
The other direction is clear.
\end{proof}

\begin{example}
The followings are eventually periodic Sturmian colorings.

\medskip
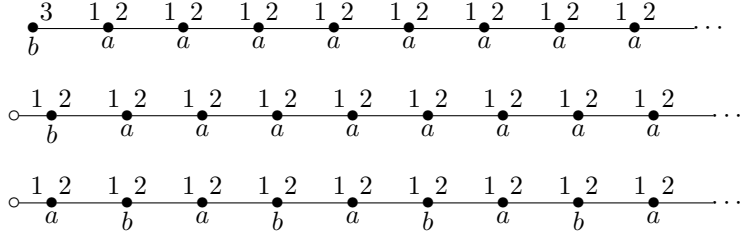
\begin{center}
\begin{tikzpicture}
  \tikzstyle{every node}=[inner sep=-1pt]
  \node (0) {$\bullet$} node [below=4pt] {$b$};
  \node (1) at (1,0) {$\bullet$} node [below=4pt] at (1,0) {$a$};
  \node (2) at (2,0) {$\bullet$} node [below=4pt] at (2,0) {$a$};
  \node (3) at (3,0) {$\bullet$} node [below=4pt] at (3,0) {$a$};
  \node (4) at (4,0) {$\bullet$} node [below=4pt] at (4,0) {$a$};
  \node (5) at (5,0) {$\bullet$} node [below=4pt] at (5,0) {$a$};
  \node (6) at (6,0) {$\bullet$} node [below=4pt] at (6,0) {$a$};
  \node (7) at (7,0) {$\bullet$} node [below=4pt] at (7,0) {$a$};
  \node (8) at (8,0) {$\bullet$} node [below=4pt] at (8,0) {$a$};
  \node (9) at (9,0) {$\cdots$};

 \path[-] (0) edge node [above=4pt] {3 \quad 1} (1)
		 (1) edge node [above=4pt] {2 \quad 1} (2)
		 (2) edge node [above=4pt] {2 \quad 1} (3)
		 (3) edge node [above=4pt] {2 \quad 1} (4)
		 (4) edge node [above=4pt] {2 \quad 1} (5)
		 (5) edge node [above=4pt] {2 \quad 1} (6)
		 (6) edge node [above=4pt] {2 \quad 1} (7)
		 (7) edge node [above=4pt] {2 \quad 1} (8)
		 (8) edge node [above=4pt] {2 \quad \ } (9);
\end{tikzpicture}
\end{center}

\medskip
\begin{center}
\begin{tikzpicture}[every loop/.style={}]
  \tikzstyle{every node}=[inner sep=-1pt]
  \node (-1) at (-.5,0) {$\circ$};
  \node (0) {$\bullet$} node [below=4pt] {$b$};
  \node (1) at (1,0) {$\bullet$} node [below=4pt] at (1,0) {$a$};
  \node (2) at (2,0) {$\bullet$} node [below=4pt] at (2,0) {$a$};
  \node (3) at (3,0) {$\bullet$} node [below=4pt] at (3,0) {$a$};
  \node (4) at (4,0) {$\bullet$} node [below=4pt] at (4,0) {$a$};
  \node (5) at (5,0) {$\bullet$} node [below=4pt] at (5,0) {$a$};
  \node (6) at (6,0) {$\bullet$} node [below=4pt] at (6,0) {$a$};
  \node (7) at (7,0) {$\bullet$} node [below=4pt] at (7,0) {$a$};
  \node (8) at (8,0) {$\bullet$} node [below=4pt] at (8,0) {$a$};
  \node (9) at (9,0) {$\cdots$};

  \path[-] (-1) edge node [above=4pt] {\ 1} (0)    
		 (0) edge node [above=4pt] {2 \quad 1} (1)
		 (1) edge node [above=4pt] {2 \quad 1} (2)
		 (2) edge node [above=4pt] {2 \quad 1} (3)
		 (3) edge node [above=4pt] {2 \quad 1} (4)
		 (4) edge node [above=4pt] {2 \quad 1} (5)
		 (5) edge node [above=4pt] {2 \quad 1} (6)
		 (6) edge node [above=4pt] {2 \quad 1} (7)
		 (7) edge node [above=4pt] {2 \quad 1} (8)
		 (8) edge node [above=4pt] {2 \quad \ } (9);
\end{tikzpicture}
\end{center}

\medskip
\begin{center}
\begin{tikzpicture}[every loop/.style={}]
  \tikzstyle{every node}=[inner sep=-1pt]
  \node (-1) at (-.5,0) {$\circ$};
  \node (0) {$\bullet$} node [below=4pt] {$a$};
  \node (1) at (1,0) {$\bullet$} node [below=4pt] at (1,0) {$b$};
  \node (2) at (2,0) {$\bullet$} node [below=4pt] at (2,0) {$a$};
  \node (3) at (3,0) {$\bullet$} node [below=4pt] at (3,0) {$b$};
  \node (4) at (4,0) {$\bullet$} node [below=4pt] at (4,0) {$a$};
  \node (5) at (5,0) {$\bullet$} node [below=4pt] at (5,0) {$b$};
  \node (6) at (6,0) {$\bullet$} node [below=4pt] at (6,0) {$a$};
  \node (7) at (7,0) {$\bullet$} node [below=4pt] at (7,0) {$b$};
  \node (8) at (8,0) {$\bullet$} node [below=4pt] at (8,0) {$a$};
  \node (9) at (9,0) {$\cdots$};

  \path[-] (-1) edge node [above=4pt] {\ 1} (0)    
		 (0) edge node [above=4pt] {2 \quad 1} (1)
		 (1) edge node [above=4pt] {2 \quad 1} (2)
		 (2) edge node [above=4pt] {2 \quad 1} (3)
		 (3) edge node [above=4pt] {2 \quad 1} (4)
		 (4) edge node [above=4pt] {2 \quad 1} (5)
		 (5) edge node [above=4pt] {2 \quad 1} (6)
		 (6) edge node [above=4pt] {2 \quad 1} (7)
		 (7) edge node [above=4pt] {2 \quad 1} (8)
		 (8) edge node [above=4pt] {2 \quad \ } (9);
\end{tikzpicture}
\end{center}

\end{example}

\begin{remark} Bi-infinite Sturmian words are defined as non-eventually periodic words with subword complexity $p(n)=n+1$. Note that bi-infinite words with $p(n)=n+1$ do not necessarily correspond to colorings of 2-regular tree with $b(n)=n+2$.
\end{remark}

\subsection{Sturmian colorings of unbounded type}

Let $\phi$ be a Sturmian coloring of unbounded type. The type set of every vertex is an infinite set. If two vertices $x,y$ have the same type set, then they have the same $n$-balls for every $n$ by Lemma~\ref{lem:2}, since there exists a sequence $n_k \to \infty$ for which $[B_{n_k}(x)]=[B_{n_k}(y)]$.

Thus we can construct $X$ by letting $X= T/\sim$, where $x \sim y$ if they have the same type set.
\begin{lemma} 
The vertices of a 1-ball may have at most three distinct type sets.
\end{lemma}

\begin{proof}
Let $x$ be the center of a 1-ball. Suppose that there are three vertices $x_1,x_2,x_3$ neighboring $x$ such that $x, x_1, x_2, x_3$ have mutually distinct type sets. 
If $n \in \Lambda_x \cap \Lambda_y$, then $[B_n(x)]=[B_n(y)]$ by the uniqueness of special $n$-ball. 
Thus for $\ell \leq n$, $\ell \in \Lambda_x$ if and only if $\ell \in \Lambda_y$. 

We conclude that the type sets of $x$ and $y$ are equal up to some number, say $N$, and they are disjoint from $N+1$.
Choose such $N$ for each pair of vertices from different classes in $B_2(x)$ and let $M$ be the maximum of such $N$'s. 
Then the type sets of two non-equivalent vertices in $B_2(x)$ intersected with $\{ M+1, M+2, \cdots \}$ are all mutually disjoint.

Now let $\ell >M$ be in the type set $\Lambda_x$. 
There exists such $\ell$ since $\Lambda_x$ is infinite. 
Since the type sets of $x, x_1, x_2, x_3$ intersected with $\mathbb{N}_{\geq M}$ are mutually disjoint, at least one of $x_1$, $x_2$, $x_3$ has a type set disjoint from $\{\ell-1, \ell, \ell+1\}$.
Denote by $x_i$ such a vertex, thus $\{\ell-1, \ell, \ell+1\} \cap \Lambda_{x_i} = \emptyset$. 

Since $\ell \in \Lambda_x$, there exists a vertex $y$ such that 
$[B_\ell (x)] = [B_\ell (y)]$ but $[B_{\ell+1} (x)] \ne [B_{\ell+1} (y)]$.
Let $f : B_\ell (x) \to B_\ell (y)$ be the color-preserving isomorphism
and let $y_i = f(x_i)$, so that $[B_{\ell-1}(x_i)] = [B_{\ell-1}(y_i)]$.
Let $p = \min \{ \Lambda_{x_i} \bigcap \mathbb{N}_{\geq \ell-1}\}  > \ell+1 $.
Then $[B_{p}(x_i)] = [B_{p}(y_i)]$, since $[B_{\ell-1}(x_i)], [B_{\ell}(x_i)], [B_{\ell+1}(x_i)]$ not being special implies that there is a unique extension from $[B_{\ell-1}(x_i)]$ to $[B_p(x_i)]$. 

Let $g : B_{p}(x_i) \to B_{p}(y_i)$ be a color-preserving isomorphism and $x' = g^{-1}(y)$.
Then $d(x', x_i) = d(y, y_i) = 1$, thus, we have $ B_{p-1} (x') \subset B_{p}(x_i)$ and 
$[ B_{p-1} (x') ] = [ B_{p-1} (y) ]$ by $g$.
From $p > \ell+1$, we have $[ B_{\ell+1} (x')] = [ B_{\ell+1} (y) ] \ne [B_{\ell+1} (x)]$  and $[ B_{\ell} (x') ] = [ B_{\ell} (y) ] = [B_\ell (x)] $,
which implies that $x$ and $x'$ are not in the same class, but their type sets have an intersection containing $\ell > M$, contradicting the choice of $M$.
\end{proof}

It follows that in the graph $X=T/\sim$, a given vertex has at most two neighboring vertices except itself. 
If there are more than one vertices with only one neighbor, then the graph $X$ is finite, which contradicts the fact that $b_\phi(n)$ is not bounded.
Therefore, we obtain the following theorem.

\begin{theorem}\label{prop:6}
If $\phi$ is a Sturmian coloring of unbounded type,
then there exists a proper quotient infinite graph $X$ and a coloring $\phi_X$ on $X$ such that
$\phi = \phi_X \circ \pi$, where $\pi$ is the projection from the regular tree $T$ to $X$.
Moreover, we have
$$VX =\{ 0,1,2, \dots , \}, \quad  EX \subset \{ [i,i+1] \, | \, i \ge 0 \} \cup \{ [i,i] \, | \, i \ge 0 \} $$
or 
$$VX =\{ \dots, -2, -1, 0,1,2, \dots , \}, \quad EX \subset \{ [i,i+1] \, | \, i \in \mathbb Z \} \cup \{ [i,i] \, | \, i \in \mathbb Z \}. $$
\end{theorem}




\begin{example}[Sturmian colorings with a periodic edge configuration]\label{ex:unbounded1}
By a given Sturmian coloring on a 2-regular tree $Y$, 
we have the following Sturmian coloring of unbounded type on $k$-regular tree.

Let $X$ be the following infinite edge-indexed colored graph: 
\medskip
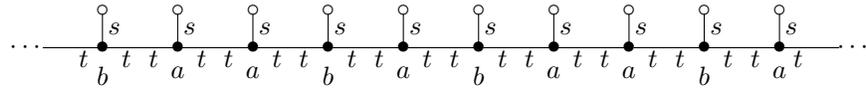
\begin{center}
\begin{tikzpicture}[every loop/.style={}]
  \tikzstyle{every node}=[inner sep=-1pt]
  \node (0) at (-6,0) {$\cdots$};
  \node (1) at (-5,0) {$\bullet$} node [below=8pt] at (-5,0) {$b$};
  \node (a1) at (-5,.5) {$\circ$};
  \node (2) at (-4,0) {$\bullet$} node [below=8pt] at (-4,0) {$a$};
  \node (a2) at (-4,.5) {$\circ$};
  \node (3) at (-3,0) {$\bullet$} node [below=8pt] at (-3,0) {$a$};
  \node (a3) at (-3,.5) {$\circ$};
  \node (4) at (-2,0) {$\bullet$} node [below=8pt] at (-2,0) {$b$};
  \node (a4) at (-2,.5) {$\circ$};
  \node (5) at (-1,0) {$\bullet$} node [below=8pt] at (-1,0) {$a$};
  \node (a5) at (-1,.5) {$\circ$};
  \node (6) at  (0,0) {$\bullet$} node [below=8pt] at (0,0) {$b$};
  \node (a6) at (0,.5) {$\circ$};
  \node (7) at  (1,0) {$\bullet$} node [below=8pt] at (1,0) {$a$};
  \node (a7) at (1,.5) {$\circ$};
  \node (8) at  (2,0) {$\bullet$} node [below=8pt] at (2,0) {$a$};
  \node (a8) at (2,.5) {$\circ$};
  \node (9) at  (3,0) {$\bullet$} node [below=8pt] at (3,0) {$b$};
  \node (a9) at (3,.5) {$\circ$};
  \node (10) at (4,0) {$\bullet$} node [below=8pt] at (4,0) {$a$};
  \node (a10) at (4,.5) {$\circ$};
  \node (11) at (5,0) {$\cdots$};

  \path[-] 
		 (0)  edge node [below=2pt] {$\ \ \ t$} (1)
		 (1)  edge node [below=2pt] {$t \ \ t$} (2)
		 (2)  edge node [below=2pt] {$t \ \ t$} (3)
		 (3)  edge node [below=2pt] {$t \ \ t$} (4)
		 (4)  edge node [below=2pt] {$t \ \ t$} (5)
		 (5)  edge node [below=2pt] {$t \ \ t$} (6)
		 (6)  edge node [below=2pt] {$t \ \ t$} (7)
		 (7)  edge node [below=2pt] {$t \ \ t$} (8)
		 (8)  edge node [below=2pt] {$t \ \ t$} (9)
		 (9)  edge node [below=2pt] {$t \ \ t$} (10)
		 (10) edge node [below=2pt] {$t \ \ \ $} (11) 
		 (1)  edge node [right=3pt] {$s$} (a1)    
		 (2)  edge node [right=3pt] {$s$} (a2)    
		 (3)  edge node [right=3pt] {$s$} (a3)    
		 (4)  edge node [right=3pt] {$s$} (a4)    
		 (5)  edge node [right=3pt] {$s$} (a5)    
		 (6)  edge node [right=3pt] {$s$} (a6)    
		 (7)  edge node [right=3pt] {$s$} (a7)    
		 (8)  edge node [right=3pt] {$s$} (a8)    
		 (9)  edge node [right=3pt] {$s$} (a9)    
		 (10) edge node [right=3pt] {$s$} (a10) ;  
\end{tikzpicture}
\end{center}
Let us index the vertex set $VX$ by $\mathbb Z$.
Each edge $[i,i+1]$ and $[i,i-1]$ are indexed with $t$ and $[i,i]$ is indexed with $s$. 
Note that $k = s+ 2t$. 

Let $\phi_{0}$ be a Sturmian coloring of $Y$. 
Since both $VY$ and $VX$ are indexed by $\mathbb Z$,
define a coloring $\phi_X$ on $X$
by $\phi_X(i) = \phi_0(i)$. 
Let $\phi$ be the coloring on $T$ given by $\phi_X \circ \pi$, where $\pi : VT \to VX$ is the projection discussed in Section~\ref{sec3}.
We claim that $b_{\phi_0} (n) = b_\phi (n)$ so that if $\phi_0$ is Sturmian, then so is $\phi$.

Let $x', y' \in \mathbb Z$ be vertices in $Y$. Let us denote the corresponding vertices in $VX$ by $x, y$,
and let $\widetilde{x} , \widetilde{y}$ be vertices in $T$ such that $\pi(\widetilde{x}) = x$, $\pi(\widetilde{y}) = y$.
If $[B_n(x')]=[B_n(y')]$ in $\phi_0$, then $\phi_0(x' +i) = \phi_0(y' +i)$ for all $-n  \le i \le n$, thus  $[B_n(\widetilde{x})]=[B_n(\widetilde{y})]$ in $\phi$.

If $[B_n(x')] \ne [B_n(y')]$ in $\phi_0$, then 
there is smallest $m, 0 \le m \le n$, such that $[B_m(x')] \ne [B_m(y')]$.
Therefore, $\phi_0(x' +i) = \phi_0(y' +i)$ for all $ |  i |  \le m-1$
and by Lemma~\ref{lem:1}.(2), $\{ [\BR_n(x',x'+1)], [\BR_n(x',x'-1)] \}$ are not $\{ [\BR_n(y',y'+1)], [\BR_n(y',y'-1)] \}$,
i.e.
\begin{multline}\label{color} \{ (\phi_0(x' + m-1), \phi_0(x' + m)), (\phi_0(x' - m+1), \phi_0(x'- m))  \} \\
\ne \{ (\phi_0(y' + m-1), \phi_0(y' + m)), (\phi_0(y' - m+1), \phi_0(y' - m))  \}. \end{multline}
In $[B_m(x)]$, there are $k(k-1)^{m-1}$  directed edges with terminal vertex of distance $m$ from $x$.
Among them, $t^{m}$ number of edges have the initial and terminal vertices colored by $\phi_0(x' + m-1), \phi_0(x' + m)$ and $t^{m}$ of them have the initial and terminal vertices colored by $\phi_0(x'- m+1), \phi_0(x '- m)$.
All the other edges have initial and terminal vertices colored by $\phi_0(x' + i)$, $|i| \le m-1$.
Therefore, by (\ref{color}),  $[B_n(\widetilde{x})] \ne [B_n(\widetilde{y})]$ in $\phi$.

Hence we have $[B_n(\widetilde{x})] = [B_n(\widetilde{y})]$  in $\phi$ if and only if $[B_n(x)] = [B_n(y)]$  in $\phi_0$.
It follows that $b_\phi (n) = b_{\phi_0} (n)$.

\end{example}


The following edge-indexed colored graph is an example of Sturmian coloring, whose coloring on the quotient ray is periodic but whose edge index is Sturmian.

\begin{example}
Let $\phi_0$ be a Sturmian coloring with colors $\{c,d\}$ on $Y$, which is associated to a bi-infinite Sturmian sequence.
Then we can construct a linear graph with loops $X$ as follows:

\medskip
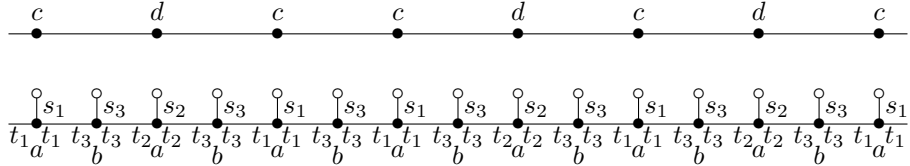
\begin{center}
\begin{tikzpicture}[every loop/.style={},scale=.8]
  \tikzstyle{every node}=[inner sep=-1pt]

  \node (b0) at (-5.5,1.5) {} ;
  \node (b1) at (-5,1.5) {$\bullet$} node [above=6pt] at (-5,1.5) {$c$};
  \node (b3) at (-3,1.5) {$\bullet$} node [above=6pt] at (-3,1.5) {$d$};
  \node (b5) at (-1,1.5) {$\bullet$} node [above=6pt] at (-1,1.5) {$c$};
  \node (b7) at (1,1.5) {$\bullet$} node [above=6pt] at (1,1.5) {$c$};
  \node (b9) at (3,1.5) {$\bullet$} node [above=6pt] at (3,1.5) {$d$};
  \node (b11) at (5,1.5) {$\bullet$} node [above=6pt] at (5,1.5) {$c$};
  \node (b13) at (7,1.5) {$\bullet$} node [above=6pt] at (7,1.5) {$d$};
  \node (b15) at (9,1.5) {$\bullet$} node [above=6pt] at (9,1.5) {$c$};
  \node (b16) at (9.5,1.5) {} ;

  \node (0) at (-5.5,0) {};
  \node (1) at (-5,0) {$\bullet$} node [below=9pt] at (-5,0) {$a$};
  \node (a1) at (-5,.5) {$\circ$};
  \node (2) at (-4,0) {$\bullet$} node [below=9pt] at (-4,0) {$b$};
  \node (a2) at (-4,.5) {$\circ$};
  \node (3) at (-3,0) {$\bullet$} node [below=9pt] at (-3,0) {$a$};
  \node (a3) at (-3,.5) {$\circ$};
  \node (4) at (-2,0) {$\bullet$} node [below=9pt] at (-2,0) {$b$};
  \node (a4) at (-2,.5) {$\circ$};
  \node (5) at (-1,0) {$\bullet$} node [below=9pt] at (-1,0) {$a$};
  \node (a5) at (-1,.5) {$\circ$};
  \node (6) at (0,0) {$\bullet$} node [below=9pt] at (0,0) {$b$};
  \node (a6) at (0,.5) {$\circ$};
  \node (7) at (1,0) {$\bullet$} node [below=9pt] at (1,0) {$a$};
  \node (a7) at (1,.5) {$\circ$};
  \node (8) at (2,0) {$\bullet$} node [below=9pt] at (2,0) {$b$};
  \node (a8) at (2,.5) {$\circ$};
  \node (9) at (3,0) {$\bullet$} node [below=9pt] at (3,0) {$a$};
  \node (a9) at (3,.5) {$\circ$};
  \node (10) at (4,0) {$\bullet$} node [below=9pt] at (4,0) {$b$};
  \node (a10) at (4,.5) {$\circ$};
  \node (11) at (5,0) {$\bullet$} node [below=9pt] at (5,0) {$a$};
  \node (a11) at (5,.5) {$\circ$};
  \node (12) at (6,0) {$\bullet$} node [below=9pt] at (6,0) {$b$};
  \node (a12) at (6,.5) {$\circ$};
  \node (13) at (7,0) {$\bullet$} node [below=9pt] at (7,0) {$a$};
  \node (a13) at (7,.5) {$\circ$};
  \node (14) at (8,0) {$\bullet$} node [below=9pt] at (8,0) {$b$};
  \node (a14) at (8,.5) {$\circ$};
  \node (15) at (9,0) {$\bullet$} node [below=9pt] at (9,0) {$a$};
  \node (a15) at (9,.5) {$\circ$};
  \node (16) at (9.5,0) {};

  \path[-] 
		 (0) edge node [below=2pt] {$ t_1$} (1)
		 (1) edge node [below=2pt] {$t_1 \ t_3$} (2)
		 (2) edge node [below=2pt] {$t_3 \ t_2$} (3)
		 (3) edge node [below=2pt] {$t_2 \ t_3$} (4)
		 (4) edge node [below=2pt] {$t_3 \ t_1$} (5)
		 (5) edge node [below=2pt] {$t_1 \ t_3$} (6)
		 (6)  edge node [below=2pt] {$t_3 \ t_1$} (7)
		 (7)  edge node [below=2pt] {$t_1 \ t_3$} (8)
		 (8)  edge node [below=2pt] {$t_3 \ t_2$} (9)
		 (9)  edge node [below=2pt] {$t_2 \ t_3$} (10)
		 (10)  edge node [below=2pt] {$t_3 \ t_1$} (11) 
		 (11)  edge node [below=2pt] {$t_1 \ t_3$} (12) 
		 (12)  edge node [below=2pt] {$t_3 \ t_2$} (13) 
		 (13)  edge node [below=2pt] {$t_2 \ t_3$} (14) 
		 (14)  edge node [below=2pt] {$t_3 \ t_1$} (15) 
		 (15)  edge node [below=2pt] {$t_1$} (16) 
 		 (1) edge node [right=3pt] {$s_1$} (a1)    
 		 (2) edge node [right=3pt] {$s_3$} (a2)  
 		 (3) edge node [right=3pt] {$s_2$} (a3)    
 		 (4) edge node [right=3pt] {$s_3$} (a4)  
 		 (5) edge node [right=3pt] {$s_1$} (a5)    
 		 (6) edge node [right=3pt] {$s_3$} (a6)  
 		 (7) edge node [right=3pt] {$s_1$} (a7)    
 		 (8) edge node [right=3pt] {$s_3$} (a8)  
		 (9) edge node [right=3pt] {$s_2$} (a9)
 		 (10) edge node [right=3pt] {$s_3$} (a10)  
 		 (11) edge node [right=3pt] {$s_1$} (a11)    
 		 (12) edge node [right=3pt] {$s_3$} (a12)  
 		 (13) edge node [right=3pt] {$s_2$} (a13)  
 		 (14) edge node [right=3pt] {$s_3$} (a14)  
 		 (15) edge node [right=3pt] {$s_1$} (a15)    
		 (b0) edge (b1)    
		 (b1) edge (b3)    
 		 (b3) edge (b5)    
 		 (b5) edge (b7)    
 		 (b7) edge (b9)    
		 (b9) edge (b11)
 		 (b11) edge (b13)    
 		 (b13) edge (b15) 
 		 (b15) edge (b16);
\end{tikzpicture}
\end{center}
Let $s_i, t_i$, $i = 1,2,3$ be integers satisfying $t_i \ge 1$, $s_i \ge 0$, $s_i + 2 t_i = k$ for each $i = 1,2,3$ and $s_1 \ne s_2$.
In $X$ each vertex is colored by $a$ and $b$ in the alternating way.
Each vertex colored by $b$ have two directed edges indexed by $t_3$ and a loop indexed by $s_3$.
Each vertex colored by $a$ corresponds either to a vertex colored by $c$ or $d$ in $Y$.
If it corresponds to a vertex colored by $c$ (respectively $d$), put indices $s_1$ and $t_1$ 
(respectively $s_2$ and $t_2$) to the neighboring edges in $X$.
By this association we have a Sturmian coloring on $T$ from a Sturmian coloring on $Y$.
By a proof similar to the previous example, it is a Sturmian coloring of $T$.
\end{example}

Finally, we have an example in which neither vertex coloring nor edge index is periodic.

\begin{example}
Let $\phi_0$ be a Sturmian coloring on $Y$ in which $a-b-a$ and $b-b-b$ are not admissible.
We can construct a linear graph with loops $X$ as follows:

\medskip
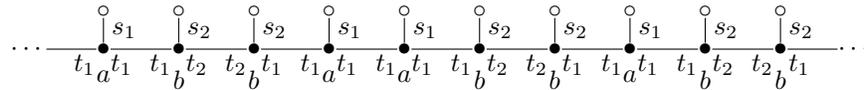
\begin{center}
\begin{tikzpicture}[every loop/.style={}]
  \tikzstyle{every node}=[inner sep=0pt]
  \node (0) at (-6,0) {$\cdots$};
  \node (1) at (-5,0) {$\bullet$} node [below=8pt] at (-5,0) {$a$};
  \node (a1) at (-5,.5) {$\circ$};
  \node (2) at (-4,0) {$\bullet$} node [below=8pt] at (-4,0) {$b$};
  \node (a2) at (-4,.5) {$\circ$};
  \node (3) at (-3,0) {$\bullet$} node [below=8pt] at (-3,0) {$b$};
  \node (a3) at (-3,.5) {$\circ$};
  \node (4) at (-2,0) {$\bullet$} node [below=8pt] at (-2,0) {$a$};
  \node (a4) at (-2,.5) {$\circ$};
  \node (5) at (-1,0) {$\bullet$} node [below=8pt] at (-1,0) {$a$};
  \node (a5) at (-1,.5) {$\circ$};
  \node (6) at (0,0) {$\bullet$} node [below=8pt] at (0,0) {$b$};
  \node (a6) at (0,.5) {$\circ$};
  \node (7) at (1,0) {$\bullet$} node [below=8pt] at (1,0) {$b$};
  \node (a7) at (1,.5) {$\circ$};
  \node (8) at (2,0) {$\bullet$} node [below=8pt] at (2,0) {$a$};
  \node (a8) at (2,.5) {$\circ$};
  \node (9) at (3,0) {$\bullet$} node [below=8pt] at (3,0) {$b$};
  \node (a9) at (3,.5) {$\circ$};
  \node (10) at (4,0) {$\bullet$} node [below=8pt] at (4,0) {$b$};
  \node (a10) at (4,.5) {$\circ$};
  \node (11) at (5,0) {$\cdots$};
  \path[-] 
		 (0) edge node [below=2pt] {$\ \ \ t_1$} (1)
		 (1) edge node [below=2pt] {$t_1 \ \ t_1$} (2)
		 (2) edge node [below=2pt] {$t_2 \ \ t_2$} (3)
		 (3) edge node [below=2pt] {$t_1 \ \ t_1$} (4)
		 (4) edge node [below=2pt] {$t_1 \ \ t_1$} (5)
		 (5) edge node [below=2pt] {$t_1 \ \ t_1$} (6)
		 (6) edge node [below=2pt] {$t_2 \ \ t_2$} (7)
		 (7) edge node [below=2pt] {$t_1 \ \ t_1$}  (8)
		 (8) edge node [below=2pt] {$t_1 \ \ t_1$}  (9)
		 (9) edge node [below=2pt] {$t_2 \ \ t_2$} (10)
		 (10) edge node [below=2pt] {$t_1 \ \ \ $} (11) 
 		 (1) edge node [right=3pt] {$s_1$} (a1)    
 		 (2) edge node [right=3pt] {$s_2$} (a2)    
 		 (3) edge node [right=3pt] {$s_2$} (a3)    
 		 (4) edge node [right=3pt] {$s_1$} (a4)
  		 (5) edge node [right=3pt] {$s_1$} (a5)
 		 (6) edge node [right=3pt] {$s_2$} (a6)    
 		 (7) edge node [right=3pt] {$s_2$} (a7)    
		 (8) edge node [right=3pt] {$s_1$} (a8) 
 		 (9) edge node [right=3pt] {$s_2$} (a9)    
 		 (10) edge node [right=3pt] {$s_2$} (a10)    ;
\end{tikzpicture}
\end{center}
Let $s_i, t_i$, $i = 1,2$ be integers satisfying $t_i \ge 1$, $s_i \ge 0$ for each $i = 1,2$ and $ 2 t_1 + s_1= k$, $t_1 +t_2 + s_2 =k$ and $s_1 \ne s_2$
(The case of $s_1 = s_2$ corresponds to Example~\ref{ex:unbounded1}).
In $X$, each vertex is colored by $a$ and $b$ according to $\phi_0$.
All vertices colored by $a$ have directed edges indexed by $t_1$  to the vertices of other class and $s_1$ to the vertices of same class. Vertices colored by $b$ have directed edges indexed by $t_1$ and $t_2$ to the vertices of different class colored by $a$ and $b$ respectively and $s_2$ to the vertices of same class.
Each vertex in $Y$ colored by $a$ and $b$ is associated to a vertex in $T$ colored by $a$ and $b$ respectively.
By this association we have a Sturmian coloring on $T$ from a Sturmian coloring on $Y$, in which there is no 1-ball of $a-b-a$ and $b-b-b$.
\end{example}

\end{document}